\pdfoutput=1
\documentclass[11pt]{amsart}
\usepackage{tabularx,booktabs,tikz}
\usepackage{caption}
\usepackage{amsmath}
\usepackage{amsfonts}

\usepackage{amscd}
\usepackage{amsthm}
\usepackage{amssymb} \usepackage{latexsym}
\usepackage{mathrsfs}
\usepackage{eufrak}
\usepackage{euscript}
\usepackage{epsfig}
\usepackage{graphics}
\usepackage{array}
\usepackage{enumerate}
\usepackage{dsfont}
\usepackage{color}
\usepackage{wasysym}
\usepackage{hyperref}
\usepackage{pdfsync}

\DeclareMathOperator{\spn}{span}
\DeclareMathOperator{\sgn}{sgn}
\DeclareMathOperator{\diag}{diag}

\newcommand{\bel}[1]{\begin{equation*}\label{#1}}
	
	\newcommand{\be}{\begin{equation}}
		
		\newcommand{\ba}{\begin{eqnarray}}
			\newcommand{\ea}{\end{eqnarray}}

		\newcommand{\qe}{\end{equation}}
	
	\newcommand{\N}{{\mathbb N}}
	
	\newcommand{\rank}{\mathrm{rank}}
	
	\newcommand{\Hmm}[1]{\leavevmode{\marginpar{\tiny%
				$\hbox to 0mm{\hspace*{-0.5mm}$\leftarrow$\hss}%
				\vcenter{\vrule depth 0.1mm height 0.1mm width \the\marginparwidth}%
				\hbox to
				0mm{\hss$\rightarrow$\hspace*{-0.5mm}}$\\\relax\raggedright #1}}}
	
	\newtheorem{theorem}{Theorem}[section]
	
	\newtheorem{lemma}[theorem]{Lemma}
	\newtheorem{corollary}[theorem]{Corollary}

	\newtheorem{remark}[theorem]{Remark}
	
	\newtheorem{prop}[theorem]{Proposition}
	
	\newtheorem{example}[theorem]{Example}

	\newcommand{\tm}{\begin{theorem}}
		\newcommand{\tmd}{\end{theorem}}
	\newcommand{\co}{\begin{corollary}}
		\newcommand{\cod}{\end{corollary}}
	\newcommand{\prp}{\begin{prop}}
		\newcommand{\prpd}{\end{prop}}
	
\begin{document}

		\title[Existence theory for exponential nonlinear equations]{Existence theory for elliptic equations of general exponential nonlinearity on finite graphs}
		
		\author{Bobo Hua}
		\address{Bobo Hua: School of Mathematical Sciences, LMNS, Fudan University, Shanghai 200433, China; Shanghai Center for Mathematical Sciences, Fudan University, Shanghai 200433, China}
		\email{bobohua@fudan.edu.cn}
		
		\author{Linlin Sun}
		\address{Linlin Sun: School of Mathematics and Computational Science, Xiangtan University, Hunan 411105, China}
		\email{sunll@xtu.edu.cn}
		
		\author{Jiaxuan Wang}
		\address{Jiaxuan Wang: School of Mathematical Sciences, Fudan University, Shanghai 200433, China}
		\email{jiaxuanwang21@m.fudan.edu.cn}

		
		\begin{abstract}

			We study semilinear elliptic equations on finite graphs with fully general exponential nonlinearities, thereby extending classical equations such as the Kazdan-Warner and Chern-Simons equations. A key contribution of this work is the development of new techniques for deriving a priori estimates in this generalized setting, which reduce the original finite graph to a graph with only two vertices. This reduction enables us to explicitly compute the Brouwer degree and to establish the existence of solutions when the degree is nonzero. Furthermore, using the method of sub- and supersolutions, we also prove the existence of solutions in cases where the Brouwer degree vanishes. 
			
			

		\end{abstract}
		\maketitle


		
		\par
		\maketitle

		\bigskip
		
			
			\section{Introduction}
			
			On a closed Riemann surface, for solving the prescribed Gaussian curvature problem Kazdan and Warner  \cite{kazdan1974curvature} initiated the study of the well-known Kazdan-Warner equation 
			\begin{equation}\label{eq:kw}
				\begin{aligned}
					-\Delta u=he^u-f,
				\end{aligned}
			\end{equation} where $h,f$ are smooth functions and $\Delta$ is the Laplace-Beltrami operator, which is a semilinear elliptic equation of critical nonlinearity. They proved the existence of solutions using variational methods. It has been extensively investigated in the literature, for instance in \cite{chen1993qualitative,chen1993gaussian,ding1997differential,ding1999existence}. Such exponential nonlinearity also appears in the well-known Chern-Simons equation
			\begin{equation}\label{eq:cs}
				\begin{aligned}
					-\Delta u=\lambda e^u\left(e^u-1\right)+4\pi\sum_{i=1}^{k_0}m_i\delta_{p_i},
				\end{aligned}
			\end{equation}	
			where $\lambda>0$, $m_i\in\mathbb{N}$, and $\delta_{p_i}$ denotes the Dirac mass at the point $p_i$. The Chern-Simons vortices were found in \cite{jaffee1980vortices,hong1990multivortex}, and people were interested in proving the existence of solutions, which are classified into topological and nontopological solutions, see e.g \cite{ronggang1991existence,spruck1995topological,caffarelli1995vortex} and \cite{chen1994nonlinear,chae2000existence,chan2002non,choe2011existence} respectively.
			


			In recent years, the study of partial differential equations on graphs has received increasing attention and has played an important role in many fields.  Grigor'yan, Lin, and Yang \cite{grigor2016kazdan} studied the Kazdan-Warner equation \eqref{eq:kw} on finite graphs, and derived the existence results of solutions, which were extended in \cite{ge2017kazdan,ge2018kazdan}. Other results on the Kazdan-Warner equation can be found, for example, in references \cite{fang2018class,ge2018note,keller2018kazdan,ge2020p,liu2020multiple,camilli2022note,pinamonti2022existence,sun2022brouwer,li2023existence,zhang2024fractional}. In \cite{huang2020existence}, Huang, Lin, and Yau proved the existence of solutions to the Chern-Simons equation \eqref{eq:cs} on finite graphs. Later, Hou and Sun \cite{hou2022existence} derived solutions to \eqref{eq:cs} for the critical case, and studied the generalized Chern-Simons-Higgs equation. The first author and third author, together with Huang, \cite{hua2023existence} proved the existence of topological solutions to \eqref{eq:cs} on (infinite) lattice graphs. For more related studies, we refer the readers to \cite{huang2021mean,lu2021existence,chao2022existence,hu2022existence,lin2022calculus,chao2023multiple,gao2023existence,hou2024topological,hou2024solutions,li2024topological,sun2024sinh}.

			In this paper, we consider the semilinear elliptic equation of general exponential nonlinearity
			\begin{equation}\label{eq:m1}
				\begin{aligned}
					-\Delta u(x)=\sum_{i=0}^nf_i(x)e^{iu(x)}
				\end{aligned}
			\end{equation}
			for a positive integer $n$ on a finite connected weighted graph $G$, where $\Delta$ is the graph Laplacian, and $f_i$ are functions on vertices. This extends the Kazdan-Warner equation, the Chern-Simons equation on graphs and others. As a special case the equation $$-\Delta u=\frac{1}{m+1}(e^{(m+1)u}-e^{2(m+1)u}),\quad m\in \N,$$ appeared in the continuous case \cite[(2.10)]{chen1994nonlinear}, while our equation is a polynomial of the exponential of full generality.

			Topological degree theory is a powerful tool in the study of partial differential equations. Recently, the second author and Wang \cite{sun2022brouwer} first applied this method on finite graphs and computed the Brouwer degree of the Kazdan-Warner equation \eqref{eq:kw}, extending the results in \cite{chen2003topological}. Utilizing this tool, Liu \cite{liu2022brouwer} studied the mean field equation, and Li, the second author and Yang \cite{li2024topological} solved the Chern-Simons-Higgs models. In this paper,  we establish the existence of solutions of \eqref{eq:m1} on a finite graph by computing the Brouwer degree and the method of sub- and supersolutions. 
			
			Consider a finite connected graph  $G=(V,E)$ with vertex set $V$ and edge set $E$. Let 
			$$m:V\rightarrow \mathbb{R}^+,\ x\mapsto m(x)$$
			be a vertex weighted function and
			$$\omega: V\times V\rightarrow [0,+\infty), \ \{x,y\}\mapsto \omega_{xy}$$
			be a weight function, which satisfies $\omega_{xy}=\omega_{yx}$ and $\omega_{xy}>0$ if and only if $\{x,y\}\in E,$ denoted by $x\sim y$. The graph $G=(V,E,\omega,m)$ is called a weighted graph. For a weighted graph $G$ and a function $u:V\rightarrow \mathbb{R}$, we define the Laplace operator by
			$$\Delta u(x)=\frac{1}{m(x)}\sum_{y\sim x}\omega_{xy}\left(u(y)-u(x)\right).$$
			For a function $f$ on $G$, we define the average of  $f$ over $V$ as
			$$\overline{f}:=	\frac{\sum_{x\in V}m(x)f(x)}{\sum_{x\in V}m(x)}.$$
			
			In this paper we study the following equation equivalent to \eqref{eq:m1}
			\begin{equation}\label{eq:maineq}
				\begin{aligned}
					-\Delta u(x)=F_n(x,u)=\sum_{i=1}^nf_i(x)e^{iu(x)}+c
				\end{aligned}
			\end{equation}
			on a weighted graph $G=(V,E,\omega,m)$, where $c$ is a constant. For convenience, we write 
			$$F_n(x,y)=\sum_{i=1}^nf_i(x)e^{iy}+c,$$
			where $x\in V,\ y\in\mathbb{R}$. We define the leading coefficient $D_{F_n}(x)$ as
			\begin{equation*}
				\begin{aligned}
					D_{F_n}(x):=\left\{
					\begin{aligned}
						f_n(x),\ \ &f_n(x)\neq0,\\
						f_k(x),\ \ &f_n(x)=\dots=f_{k+1}(x)=0,\  f_k(x)\neq0,\\
						0,\ \ &f_n(x)=\dots=f_{1}(x)=0,
					\end{aligned}
					\right.
				\end{aligned}
			\end{equation*}
			and set
			$$q(F_n):=\max_{x\in V}\lim_{y\rightarrow+\infty}F_n(x,y).$$ We distinguish the cases with $q(F_n)<+\infty$ and $q(F_n)=+\infty$ as follows:
			
			\begin{enumerate}[(1)]
				\item  $q(F_n)<+\infty$ if and only if $D_{F_n}\leqslant0$.   
				\item $q(F_n)=+\infty$ if and only if $D_{F_n}$ is positive somewhere.  
				
			\end{enumerate}

			First, we prove the following theorem, which is a discrete analog of Brezis and Merle's result \cite{brezis1991uniform}.
			
			\tm\label{tm:alternative}
			Let $G=(V,E,\omega,m)$ be a finite connected weighted graph. For  $m\in\mathbb{N},$ let $u_m$ be a solution to the following equation on $G$
			$$-\Delta u_m(x)=\sum_{i=1}^nf_{i,m}(x)e^{iu_m(x)}+c_m,$$
			and 
			$$\lim_{m\rightarrow+\infty}f_{i,m}=f_i,\ \lim_{m\rightarrow+\infty}c_m=c.$$
			Then there exists a subsequence $\{u_{m_k}\}$, which satisfies the following alternative:
			
			\begin{enumerate}[(1)]
				\item either $\{u_{m_k}\}$ is uniformly bounded, or            
				\item $u_{m_k}\rightarrow-\infty$ uniformly on $G$, or
				
				\item $u_{m_k}\rightarrow+\infty$ uniformly on $G$.
			\end{enumerate}
			\tmd
			
			Under certain assumptions, an a priori estimate can be derived, implying that the first result of Theorem ~\ref{tm:alternative} holds. 
			
			\tm\label{tm6.1}
			Let $u$ be a solution to \eqref{eq:maineq} on a finite connected weighted graph $G=(V,E,\omega,m)$. Assume that 
			$$\sum_{i=1}^n\max_{x\in V}|f_i(x)|+|c|\leqslant K,$$	    
			and either $$|c|\geqslant K^{-1}$$
			or
			$$c=0, \ |\overline{f_1}|\geqslant K^{-1}$$ for some constant $K>0$.
			Suppose that there exists some vertex $x_0\in V$ such that one of the following conditions holds:		     
			\begin{enumerate}[(i)]
				\item $D_{F_n}(x_0)\geqslant\frac{1}{K}$.
				\item $D_{F_n}(x_0)\leqslant-\frac{1}{K}$. Moreover, for every $x\in V$,  $f_i(x)\leqslant0$ for every $1\leqslant i\leqslant n$, or
				$D_{F_n}(x)\leqslant-\frac{1}{K}$.

			\end{enumerate}
			Then there exists a constant $C_n$ depending only on $K,V,\omega,n$ such that the solution $u$ satisfies 
			$$\max_{x\in V}|u(x)|\leqslant C_n.$$
			\tmd
			
			For the proof of the theorem, since the nonlinearity of the equation is very general and complicated, we develop a new method to classify different cases, leading to a priori estimate. We use the signature of $D_{F_n}$ at a single vertex $x_0$ and the structure of the polynomial in the nonlinearity, i.e. the asymptotics of $F_n(x,y)$  as $y\to +\infty$, to control the upper bound of the function $u.$
			
			By  Theorem ~\ref{tm6.1}, the following corollary holds.
			
			\begin{corollary}\label{co1.3}
				Consider equation \eqref{eq:maineq} on a finite connected weighted graph $G=(V,E,\omega,m)$. We assume that $\overline{f_1}\neq0$ when $c=0$. Then the solutions to \eqref{eq:maineq} are uniformly bounded.
			\end{corollary}
			
			An important fact follows from Corollary ~\ref{co1.3} that the Brouwer degree is well-defined. The Brouwer degree associated with equation \eqref{eq:maineq} is denoted by $d_{f_n,\dots, f_1,c}$, and see Subsection ~\ref{sub2.2} for details.
			
			In each case of Theorem \ref{tm6.1}, by Lemma ~\ref{lm3} we can reduce the graph to a graph with two vertices, without changing the value of the Brouwer degree. This key reduction makes the computation of the topological degree $d_{f_n,\dots, f_1,c}$ feasible.  Each case can be further reduced to conditions involving only $c,\overline{f_1}$ and the polynomial $F_n(x,y),$ and we prove the following theorem.
			
			\tm\label{1.2}
			Suppose that $u$ is a solution to \eqref{eq:maineq} on a finite connected weighted graph $G=(V,E,\omega,m)$, and assume that $\overline{f_1}\neq0$ when $c=0$. Then we have
			\begin{equation*}
				\begin{aligned}
					d_{f_n,\dots, f_1,c}=\left\{
					\begin{aligned}
						1,\ \ \   &q(F_n)<+\infty,\ \mathrm{with\ either}\  c>0,\ \text{or}\ c=0 \ \text{and}\ \overline{f_1}>0, \\
						-1,\ \ \ &q(F_n)=+\infty,\ \mathrm{with\ either}\ c<0,\ \text{or}\ c=0 \ \text{and}\ \overline{f_1}<0,\\
						0,\ \ \ &\mathrm{otherwise}.\\
					\end{aligned}
					\right.
				\end{aligned}
			\end{equation*}
			
			\tmd
			
			By the Kronecker existence, it immediately follows that if $d_{f_n,\dots, f_1,c}\neq0$, then there exists at least one solution to \eqref{eq:maineq}. In the remaining cases, we assume that there exists some $f_i$ (with $1\leqslant i\leqslant n$) that is positive somewhere. Otherwise, it is easy to obtain the nonexistence of solutions to \eqref{eq:maineq}.  We divide the remaining cases into two cases.
			\begin{enumerate}[(A)]
				\item $q(F_n)=+\infty$, with either $c>0,\ \text{or}\ c=0 \ \text{and}\ \overline{f_1}\geqslant0.$
				
				\item $q(F_n)<+\infty$, with either $c<0,\ \text{or}\ c=0 \ \text{and}\ \overline{f_1}\leqslant0.$
			\end{enumerate}		     			     
			Then we derive the following existence result.

			\tm\label{1.3}
			Let $G=(V,E,\omega,m)$ be a finite connected weighted graph. Suppose that in equation \eqref{eq:maineq}, at least one $f_i$, $1\leqslant i\leqslant n$, is positive somewhere. There exists a positive constant $c_n$ depending on each $f_i$, $1\leqslant i\leqslant n$, and for each $2\leqslant i\leqslant n$, there exists a function $f_i^*$ depending on other $f_j$, $1\leqslant j\leqslant n$ and $j\neq i$, such that equation \eqref{eq:maineq} is solvable if one of the following conditions holds:
			\begin{enumerate}[($i^*$)]

				\item $\overline{f_1}<0$ in case $(A)$ and $c\in (0,c_n]$.
				
				\item $\overline{f_1}>0$ in case $(B)$ and $c\in [-c_n,0)$.
				
				\item $\overline{f_1}\geqslant0$ in case $(A)$, $f_i\leqslant f_i^*$ for some integer $2\leqslant i\leqslant n$, and $c\in [0,c_n]$.
				
				\item $\overline{f_1}\leqslant0$ in case $(B)$, $f_i\geqslant f_i^*$ for some integer $2\leqslant i\leqslant n-1$, and $c\in [-c_n,0]$.
				
			\end{enumerate}
			Further, assume that one of the above conditions holds. Then the equation admits at least two solutions if $c\in (-c_n,0)$ in case $(A)$ or $c\in (0,c_n)$ in case $(B)$.
			\tmd
			\begin{remark}
				In Theorem ~\ref{1.3}, the functions $f_i^*$ contribute to understanding when solutions can be expected to exist for the last two conditions. In these cases, the equation doesn't always admit a solution, as illustrated in Example~\ref{ex1}, ~\ref{ex2}, ~\ref{ex3}. 
			\end{remark}		    

			Let us introduce the proof strategies of our main theorems. First, we prove Theorem ~\ref{tm:alternative}.  It suffices to consider the case $\{u_m\}$ is not uniformly bounded. We divide into two cases. If there exists a subsequence $\{u_m\}$ has a uniform upper bound, then by equation \eqref{eq:maineq}, $|\Delta u_m|$ is uniformly bounded, which yields that $u_m\rightarrow-\infty$. This proves the assertion (2) in the theorem.  If any subsequence $\{u_m\}$ doesn't have a uniform upper bound,  then we want to prove the assertion (3). It is noted that in this assumption, by equation \eqref{eq:maineq}$, |\Delta u_{m,-}|$ is uniformly bounded, where $u_{m,-}=\max\{-u_m,0\}$. This implies that $\{u_m\}$ has a uniform lower bound. Then
			we consider $\{y_m\}\subset V$ such that 
			$$u_m(y_m)=\min_{y\in V}u_m(y).$$
			If $\{u_m(y_m)\}$  has a uniform upper bound, then the uniform boundedness of $\{u_m(y_m)\}$ can be extended from one vertex to its neighbors, and further to the entire graph by the connectedness, which contradicts our assumption. Thus, $\{u_m\}$ must tend to $+\infty$ uniformly.
			
			For the proof of Theorem  ~\ref{tm6.1}, we use Theorem~\ref{tm:alternative}, and it suffices to rule out the latter two cases in Theorem~\ref{tm:alternative}. Based on the assumptions, it is not difficult to obtain a lower bound for $u$. Then we can assume a solution subsequence $\{u_m\}$ which tends to $+\infty$ uniformly.  For case $(i)$, we consider equation \eqref{eq:maineq} on $x_0$, and we immediately obtain a contradiction. For case $(ii)$, note that $F_n(x,u_m)$ is uniformly bounded, which implies that $|\Delta u_m|$ is uniformly bounded on $G$. Considering equation \eqref{eq:maineq} on $x_0$,  we also obtain a contradiction.
			
			To prove Theorem~\ref{1.2}, our key approach is Lemma~\ref{lm3}. This lemma allows us to reduce the set of vertices satisfying $f_1=f_2=\dots=f_n=0$ without changing the Brouwer degree. For simplicity, we only introduce the proof for the case where $c\neq0$. By Theorem~\ref{tm6.1}, we can make the homotopy transformation in each case. Note that for case $(i)$ and $f_i\leqslant0$ in case $(ii)$, a priori estimates only impose restrictions on certain $f_i(x_0)$ which attains $D_{F_n}(x_0)$. Thus, we 
			transform $f_j$ (for $i\neq j$) to be zero, and transform $f_i$ to be zero on all vertices except $x_0$. Applying Lemma~\ref{lm3}, we can reduce the graph to a two-vertex graph, which is crucial for the computation of the Brouwer degree. For case $(ii)$, observe that the conditions don't impose any restrictions on the vertices where $f_i\geqslant0$. Thus, by transforming the values on these vertices to zero, we are reduced to the case where $f_i\leqslant0$, $1\leqslant i\leqslant n$. 
			
			Finally, substantial effort is devoted to establishing the existence results in Theorem ~\ref{1.3}. We begin by analyzing  the existence of solutions when $c$ is sufficiently close to zero. The key step in the proof is to construct a supersolution for case $(A^*)$ and a subsolution for case $(B^*)$. For conditions $(i^*)$ and $(ii^*)$, we choose a function $v$ satisfying
			$$\Delta v=\sum_{i=1}^na^{i-1}\left(f_i-\overline{f_i}\right),$$
			where $a>0$ is a sufficiently small constant. Consider the function $av+\ln a$, which can serve as a supersolution for case $(A^*)$ and a subsolution for case $(B^*)$ if $c$ is sufficiently small. For conditions $(iii^*)$ and $(iv^*)$, we consider a solution $u^*$ to the equation
			$$-\Delta u^*=He^{u^*},$$
			where $H$ changes sign and $\overline{H}<0$. Note that equation \eqref{eq:maineq} is equivalent to 
			$$-\Delta u=\sum_{i=1}^nf_ie^{-ik}e^{iu}+c,$$
			where $k$ is a constant. By adjusting the parameter $k$ and choosing $c$ sufficiently small, we can let 
			$$\sum_{j=1}^{n}f_je^{-ik}e^{(i-1)u^*}+ce^{-u^*}=H,$$
			where $f_i=f_i^*$ for a certain $i$. The function $f_i^*$ is determined by the above equality and thus depends on $H$. As a result, $u^*$ can serve as a supersolution for case $(A^*)$ and a subsolution for case $(B^*)$.
			Moreover, we analyze the behavior of $c_n$ and the number of solutions. If $c$ is sufficiently large, then for case $(A)$, we consider the vertex where $u$ attains its minimum. It then follows that 
			$u$ must be bounded from below; for instance, we may assume  $u\geqslant 0$. Consider the equation on the vertex $x_0$, where $x_0$ is mentioned in Theorem~\ref{tm6.1}, and we find the equation fails to hold for large $c$. For case $(B)$, for sufficiently negative $c$ we observe that $F_n(x,u)<0$, which implies nonexistence of solutions. This ensures that $c_n$ is bounded. If $0<c<c_n$, we assume that there exists only one solution. Now we introduce the singular homology theory and use $q$-th critical group in \cite{chang1993infinite}. Theorem 3.2 in \cite{chang1993infinite} yields that $d_{f_n,\dots,f_1,c}=1$, which contradicts the condition $d_{f_n,\dots,f_1,c}=0$.
			
			The paper is organized as follows. In Section~\ref{se2}, we introduce the setting of graphs and topological degree, and we prove some key lemmas. In Section~\ref{se3}, we establish Theorem~\ref{tm:alternative}. In the remaining part of Section~\ref{se3}, and Section ~\ref{s4}, \ref{se5}, we consider the case $n=2$, and we prove Theorem~\ref{tm6.1} in Section~\ref{se3}, Theorem~\ref{1.2} in Section~\ref{s4}, Theorem~\ref{1.3} in Section~\ref{se5}. In Section~\ref{se6}, we extend the results to general $n$, and complete the proof of the main theorems.
			
			\textbf{Acknowledgements.} We thank Genggeng Huang for helpful discussions and suggestions for the paper. B. Hua is supported by NSFC, No. 12371056. B. Hua and J. Wang are supported by Shanghai Science and Technology Program [Project No. 22JC1400100].

			\section{Preliminaries}\label{se2}
			\subsection{The setting of graphs and some lemmas}
			\
			In this subsection, we introduce the setting of weighted graphs, and list some useful lemmas which are discrete versions of some well-known results.
			
			We consider a finite, connected, weighted graph $G=(V,E,\omega,m)$.  For convenience, we write
			$$\sum_{V}f:=\sum_{x\in V}m(x)f(x).$$
			
			For a function $u:V\rightarrow \mathbb{R}$, we define the Laplace operator by
			$$\Delta u(x)=\frac{1}{m(x)}\sum_{y\in V}\omega_{xy}\left(u(y)-u(x)\right).$$
			Green's identity holds on $G$
			$$-\sum_{V}\Delta uv= \sum_{V}\Gamma(u,v)=\frac{1}{2} \sum_{x,y\in V}\omega_{xy}\left(u(y)-u(x)\right)\left(v(y)-v(x)\right),$$
			where 
			$$\Gamma(u,v)(x)=\frac{1}{2} \sum_{y\in V}\frac{\omega_{xy}}{m(x)}\left(u(y)-u(x)\right)\left(v(y)-v(x)\right).$$
			We denote 
			$$|\nabla u|(x)=\sqrt{\Gamma(u,u)(x)}.$$
			For any function $u:V\rightarrow \mathbb{R}$, we define the $l^p$ norm by
			\begin{equation*}
				\|u\|_{l^p(V)}=\left\{
				\begin{aligned}
					& \left(\sum_{V}|u|^p\right)^{\frac{1}{p}},\ 1\leqslant p<\infty,\\
					& \sup_{x\in V}|u(x)|, \ p=\infty.\\
				\end{aligned}
				\right.
			\end{equation*}
			Thus we can define the Sobolev space $W^{1,p}(V)$ with the norm
			$$\|u\|_{W^{1,p}(V)}:=\|u\|_{l^{p}(V)}+\||\nabla u|\|_{l^{p}(V)}.$$
			
			In this paper, we study the following semilinear equation of general exponential nonlinearity
			\begin{equation*}
				\begin{aligned}
					-\Delta u(x)=F_n(x,u)=\sum_{i=0}^nf_i(x)e^{iu(x)}
				\end{aligned}
			\end{equation*}
			on $G$, where $f_n$ is not always zero. We can assume that  $f_0$ equals to a constant $c=\overline{f_0}$, where $$\overline{f_0}=\frac{1}{|V|}\sum_{V}f_0,$$
			and 
			$$|V|=\sum_{V}1=\sum_{x\in V}m(x).$$
			Since by replacing $u$ with $w=u-v$, where $v$ satisfies
			$$-\Delta v=f_0-\overline{f_0},$$
			we have the new equation
			$$-\Delta w(x)=\sum_{i=1}^n\widetilde{f_i}(x)e^{iw(x)}+c, $$
			where $\widetilde{f_i}=f_ie^{2v}$, and has the same sign as $f_i$. Thus, in the following part, we only consider $f_0$ to be a constant.
			
			To prove the main results, we need the following lemmas.
			
			\begin{lemma}[Elliptic estimate]\label{lm1}
				For any function $u:V\rightarrow \mathbb{R}$, there exists a constant $C_1$ such that
				$$\max_{x\in V}u(x)-\min_{x\in V}u(x)\leqslant C_1\max_{x\in V}|\Delta u(x)|,$$
				where $C_1$ depends only on $|V|,\omega,m$.
			\end{lemma} 
			
			This lemma is proved in \cite{li2024topological,sun2022brouwer}. Note that in a finite-dimensional space, different norms are equivalent, so it is not difficult to derive the above lemma.
			
			To study the existence of the solution to \eqref{eq:maineq}, we also need the method of sub- and supersolutions. Consider the equation
			\begin{equation}\label{eq:general}
				\begin{aligned}
					-\Delta u=f(x,u)
				\end{aligned}
			\end{equation}
			on $G$, where 
			$$f:V\times \mathbb{R}\rightarrow \mathbb{R}$$
			is a smooth function, and we define the following functional
			$$J(u)=\sum_{V}\left(\frac{1}{2}|\nabla u|^2-F(\cdot,u(\cdot))\right),$$ 
			where $\frac{\partial F}{\partial u}=f$. We call a function $\phi$ supersolution (resp. subsolution) of this equation, if 
			$$-\Delta \phi(x)\geqslant(resp. \leqslant)f(x,\phi(x)).$$
			\begin{lemma}[the method of sub- and supersolutions]\label{lm2}
				Suppose that $\phi_1,\phi_2$ are subsolution and supersolution of \eqref{eq:general}, which satisfy $\phi_1\leqslant \phi_2$. Then any minimizer of $J$ in $\{u:V\rightarrow\mathbb{R}:\phi_1\leqslant u\leqslant \phi_2\}$ solves \eqref{eq:general}.
			\end{lemma}
			This lemma is well-known, and its discrete version is used in works such as \cite{grigor2016kazdan,huang2020existence,sun2022brouwer}.
			
			\subsection{Topological degree}\label{sub2.2}
			
			In this subsection, since we work in a finite-dimensional space, we introduce the Brouwer degree, and prove a lemma which implies that we can remove the vertices where $f_1=f_2=\dots=f_n=0$.
			
			Let $X=l^{\infty}(V)$. We define 
			$$H_{f_n,\dots,f_1,c}:X\rightarrow X,\ H_{f_n,\dots,f_1,c}(u)=\Delta u+\sum_{i=1}^nf_ie^{iu}+c.$$
			Denote $B_R= \{u\in X: \|u\|_{l^\infty}\leqslant R\}$. If we have an a priori estimate, which implies that the solutions to \eqref{eq:maineq} are uniformly bounded for suitable $f_i$, it is known that the Brouwer degree 
			$$\deg(H_{f_n,\dots,f_1,c},B_R, 0)$$
			is well-defined for large $R$, and the homotopy invariance holds. Here we let
			$$d_{f_n,\dots,f_1,c}=\lim_{R\rightarrow+\infty}\deg(H_{f_n,\dots,f_1,c},B_R, 0).$$
			Consider the functional
			$$J_{f_n,\dots,f_1,c}(u)=\sum_{V}\left(\frac{1}{2}|\nabla u|^2-\sum_{i=1}^n\frac{1}{i}f_ie^{iu}-cu\right)$$
			for $u\in W^{1,2}(V)$. If every critical point of $J_{f_n,\dots,f_1,c}$ is nondegenerate, and $\partial B_R\cap H_{f_n,\dots,f_1,c}^{-1}(\{0\})=\varnothing$, we have
			$$\deg(H_{f_n,\dots,f_1,c},B_R, 0)=\sum_{u\in B_R,H_{f_n,\dots,f_1,c}(u)=0}\sgn \det\left(DH_{f_n,\dots,f_1,c}(u)\right).$$
			The Kronecker existence theorem implies that there exists at least one solution to \eqref{eq:maineq} if $d_{f_n,\dots,f_1,c}\neq0$. For more detailed definition and properties about the Brouwer degree, see \cite{chang2005methods}.
			
			In the following part we prove a key lemma in computing the Brouwer degree. It implies that we can remove the vertices where $f_1=f_2=\dots=f_n=0$, and reconstruct a weighted connected finite graph without changing the value of the Brouwer degree. This method is used in \cite{sun2022brouwer}, and here we provide a detailed proof.
			
			\begin{lemma}\label{lm3}
				Let $G=(V,E,m,\omega)$ be a finite connected weighted graph, and consider equation \eqref{eq:m1} on $G$. We assume that $V=\{1,2,\dots,k\}$, and 
				$f_1(j)=f_2(j)=\dots=f_n(j)=0$ if and only if $ r+1	\leqslant j\leqslant k$,	where $1\leqslant r\leqslant k-1$. Then we can construct a new weighted connected finite graph $\widetilde{G}=(\widetilde{V},\widetilde{E},\widetilde{m},\widetilde{\omega})$, where $\widetilde{V}=\{1,\dots,r\}$. On $\widetilde{G}$ we consider the new equation
				$$-\Delta u= \sum_{i=0}^n\widetilde{f_i}(x)e^{iu(x)},$$
				where for $1\leqslant i\leqslant n$,
				$$\widetilde{f_i}=f_i|_{\widetilde{V}},$$
				and $\widetilde{f_0}$ satisfies $\sum_{V}\widetilde{f_0}=\sum_{V}f_0$. We have
				$$d_{\widetilde{f_n},\dots,\widetilde{f_1},\widetilde{f_0}}=d_{f_n,\dots,f_1,f_0}.$$
			\end{lemma}
			
			\begin{proof}
				Let $L=(l_{ij})$ be a symmetric matrix, where $l_{ij}=-\omega_{ij}$ for $i\neq j$, and
				$$l_{ii}=\sum_{j\neq i}\omega_{ij}.$$
				In fact the matrix $L$ is equal to the operator $-\Delta$ with $m(x)\equiv1$. For convenience, we write $f_i(j)=f_{i,j}$ and $u(j)=u_j$. Then we have
				\begin{equation}\label{Lu}
					\begin{aligned}
						L
						\begin{pmatrix}
							u_1\\
							u_2\\
							\vdots\\
							u_k
						\end{pmatrix}
						=
						\begin{pmatrix}
							m_1\sum_{i=0}^nf_{i,1}e^{iu_1}\\
							m_2\sum_{i=0}^nf_{i,2}e^{iu_2}\\
							\vdots\\
							m_k\sum_{i=0}^nf_{i,k}e^{iu_k}
						\end{pmatrix}
						.
					\end{aligned}
				\end{equation}
				We write 
				\begin{equation*}
					\begin{aligned}
						L=\left(
						\begin{matrix}
							P & Q^T\\
							Q & R\\
						\end{matrix}
						\right),
					\end{aligned}
				\end{equation*}
				and we claim that $R$ is positive definite. If not, since the sum of each row of $R$ is nonnegative, it is known that the principal minors of $R$ are all nonnegative, and $R$ is positive semi-definite. Suppose that $Y$ is a eigenvector corresponding to the eigenvalue $0$ of $R$. Thus for any $ r\times 1$ vector $X$, let $Z=\left(
				\begin{matrix}
					X \\
					Y \\
				\end{matrix}
				\right)$, and we have
				\begin{equation*}
					\begin{aligned}
						Z^TLZ= X^TPX+2X^TQ^TY\geqslant 0.
					\end{aligned}
				\end{equation*}
				Replacing $X$ with $tX$ and letting $t\rightarrow0^+$, we derive that
				$$X^TQ^TY\geqslant 0,$$
				which implies that $Q^TY=0$. Thus, $\left(
				\begin{matrix}
					0 \\
					Y \\
				\end{matrix}
				\right)$ is an eigenvector of $L$ corresponding to the eigenvalue $0$. However, it is known that $\dim \ker L=1$, and 
				$$\ker L= \spn (1,1,\dots,1)^T.$$
				This implies that $Y=0$, which is a contradiction.
				
				Moreover, by the results in Chapter 6 of \cite{berman1994nonnegative}, every element of $R^{-1}$ is nonnegative. Denote by $\rho(A)$ the spectral radius of $A$. Since $R^{-1}$ is positive definite, we can write 
				$$R=sI-A,$$
				where $s>\rho(A)$ and every element of $A$ is nonnegative. We have
				$$R^{-1}=s^{-1}\left(I-\frac{A}{s}\right)^{-1}=s^{-1}\sum_{i=0}^{\infty}\left(\frac{A}{s}\right)^i,$$
				which yields that any element of $R^{-1}$ is nonnegative.
				
				Thus, we define the $r\times 1$ vector $U_1=\left(
				\begin{matrix}
					u_1 \\
					\vdots\\
					u_r \\
				\end{matrix}
				\right)$ and the $(k-r)\times 1$ vector $U_2=\left(
				\begin{matrix}
					u_{r+1} \\
					\vdots\\
					u_k \\
				\end{matrix}
				\right)$, then equation \eqref{Lu} is equivalent to
				\begin{equation*}
					\left\{
					\begin{aligned}
						& PU_1+Q^TU_2 = H_1:=\left(
						\begin{matrix}
							m_1\sum_{i=0}^nf_{i,1}e^{iu_1}\\
							\vdots\\
							m_r\sum_{i=0}^nf_{i,r}e^{iu_r}
						\end{matrix}
						\right),\\
						& QU_1+RU_2= H_2:=\left(
						\begin{matrix}
							m_{r+1}f_{0,r+1}\\
							\vdots\\
							m_nf_{0,k}
						\end{matrix}
						\right).\\
					\end{aligned}
					\right.
				\end{equation*}
				Since $R$ is reversible, it is not difficult to see that we only need to solve for $U_1$, and the above equations imply
				\begin{equation}\label{equv:Lu}
					\begin{aligned}
						\left(P-Q^TR^{-1}Q\right)U_1=H_1-Q^TR^{-1}H_2.
					\end{aligned}
				\end{equation}
				
				Let $\widetilde{L}=P-Q^TR^{-1}Q=(\widetilde{l_{ij}})$ be a symmetric matrix,  $\widetilde{V}=\{1,\dots,r\}$,  $\widetilde{m}=m|_{\widetilde{V}}$, $\widetilde{\omega}_{ij}=-\widetilde{l}_{ij}$ for $i\neq j$, and $\widetilde{E}=\{\{x,y\}\subset \widetilde{V}\times\widetilde{V}: \widetilde{\omega}_{xy}>0\}$. Note that $R^{-1}\geqslant0$ and $Q\leqslant0$, and we deduce that $Q^TR^{-1}Q\geqslant 0$ and $\widetilde{\omega}_{ij}\geqslant0$ for $i\neq j$. 
				Define the $1\times r$ vector $X_1=(1,1,\dots,1)$ and the $1\times (k-r)$ vector $X_2=(1,1,\dots,1)$. Since the sum of each column of $L$ is equal to zero, one easily sees that 
				$$X_1Q^T+X_2R=0.$$
				This implies that the sum of each column of $Q^TR^{-1}$ is equal to $-1$. Thus, the sum of each column of $\widetilde{L}$ is zero, which implies that the new weighted finite graph $\widetilde{G}=(\widetilde{V},\widetilde{E},\widetilde{m},\widetilde{\omega})$ is well-defined. Moreover, note that 
				$$\dim \ker \widetilde{L}=1.$$
				By swapping rows and columns of $\widetilde{L}$, we can't write 
				$$\widetilde{L}=\left(
				\begin{matrix}
					A & 0\\
					0 & B\\
				\end{matrix}
				\right),$$
				otherwise we will have 
				$$\dim \ker \widetilde{L}\geqslant \dim \ker A+\dim \ker B\geqslant2,$$
				which contradicts $\dim \ker \widetilde{L}=1$. Hence,  the graph $\widetilde{G}$ is connected. 
				
				Let
				$$\left(
				\begin{matrix}
					\widetilde{f_0}(1)\\
					\vdots\\
					\widetilde{f_0}(r)
				\end{matrix}
				\right)=\left(
				\begin{matrix}
					f_{0,1}\\
					\vdots\\
					f_{0,r}
				\end{matrix}
				\right)-Q^TR^{-1}\left(
				\begin{matrix}
					f_{0,r+1}\\
					\vdots\\
					f_{0,n}
				\end{matrix}
				\right).$$
				We deduce that 
				$$\sum_{V}\widetilde{f_0}=\sum_{V}f_0.$$
				Thus, equation \eqref{equv:Lu} is equivalent to the following equation on the reduced graph $\widetilde{G}$
				$$-\Delta u= \sum_{i=0}^n\widetilde{f_i}(x)e^{iu(x)}.$$
				Note that
				\begin{equation*}
					\begin{aligned}
						&\det\left(L-\diag \left( \sum_{i=1}^nif_{i,1}e^{iu_1},\dots,\sum_{i=1}^nif_{i,k}e^{iu_k}\right)\right)\\
						=& \det R\cdot \det\left(\widetilde{L}-\diag \left(\sum_{i=1}^nif_{i,1}e^{iu_1},\dots,\sum_{i=1}^nif_{i,r}e^{iu_r}\right)\right).
					\end{aligned}
				\end{equation*}
				This implies that
				$$d_{\widetilde{f_n},\dots,\widetilde{f_1},\widetilde{f_0}}=d_{f_n,\dots,f_1,f_0}.$$
				
			\end{proof}

			\section{A Priori Estimates}\label{se3}
			In this section we first prove Theorem~\ref{tm:alternative} for general $n$.
			
			\begin{proof}[Proof of Theorem~\ref{tm:alternative}]
				We assume that $\{u_m\}$ is not uniformly bounded. If after passing to a subsequence still denoted by $\{u_m\}$, 
				$$\sup_{m}\max_{x\in V}u_m(x)<+\infty,$$ 
				then $F_n(x,u_m)$ is uniformly bounded, which yields that $|\Delta u_m|$ is uniformly bounded. This derives that
				$$\sup_{m}\left(\max_{x\in V} u_m(x)-\min_{x\in V} u_m(x)\right)<\infty,$$
				and $u_{m}\rightarrow-\infty$ uniformly on $G$. 
				
				If we have
				\begin{equation}\label{assump}
					\begin{aligned}
						\lim_{m\rightarrow+\infty}\max_{x\in V}u_m(x)=+\infty,
					\end{aligned}
				\end{equation}
				then we can assume that for every $m\in\mathbb{N}$,
				$$\max_{x\in V}u_m(x)>0.$$
				By the conditions we may assume without loss of generality that 
				$$\sum_{i=1}^n\max_{x\in V}|f_{i,m}(x)|+|c_m|\leqslant K,$$
				where $K$ is a constant independent of $m$.
				We have the following inequality for any $x\in V$,
				$$-\Delta u_{m,-}\leqslant -F_n(x,u_m)\chi_{\{u_m<0\}}\leqslant K,$$
				where we denote by $u_{m,-}(x)=\max\{-u_m(x),0\}$. 
				This implies that
				$$|\Delta u_{m,-}|\leqslant c(K)$$
				and
				$$\max_{x\in V} u_{m,-}(x)-\min_{x\in V} u_{m,-}(x)\leqslant c'(K),$$
				where $c(K)$ and $c'(K)$ are constants depending only on $G,K$. This shows $u_m$ has a uniformly lower bound. We choose $\{y_m\}$ in $V$ such that
				$$u_m(y_m)=\min_{y\in V}u_m(y).$$
				If 
				$u_m(y_m)$ is uniformly bounded, then by equation \eqref{eq:maineq}, 
				$$|\Delta u_m(y_m)|\leqslant K\left(\sum_{i=1}^ne^{i|u_m(y_m)|}+1\right)<+\infty.$$
				This means $|\Delta u_m(y_m)|$ is also uniformly bounded. Since $\Delta$ is a local operator and we already establish a lower bound, we obtain that
				$$\Delta u_m(y_m)\geqslant \frac{1}{m(y_m)}\sum_{y\sim y_m}\omega_{yy_m}u_m(y) -c',$$
				where $c'$ is a constant depending only on $G$ and the uniform bound of $u_m(y_m)$. This yields that every neighbor $y$ of $y_m$ satisfies that 
				$u_m(y)$ is uniformly bounded. Further, we can repeat the above process, and one easily sees that the uniform boundedness can be extended to the entire graph, which is a contradiction to our assumption \eqref{assump}.
			\end{proof}
			
			In the following we consider $n=2$, and derive a priori estimates for equation \eqref{eq:maineq}.
			Without loss of generality, we can assume that $m(x)\equiv1$, and one can check the influence of $m$ in the proof. First we prove an a priori estimate in the case $c\neq0$.
			
			\begin{theorem}\label{tm1}
				Let $u$ be a solution of \eqref{eq:maineq}, and there exists a constant $K>1$ satisfying
				$$\max_{x\in V} |f_2(x)|+\max_{x\in V} |f_1(x)|+|c|\leqslant K,$$
				$$|c|\geqslant \frac{1}{K}.$$
				Suppose that there exists some $x_0\in V$ such that one of the following conditions holds:
				\begin{enumerate}[(a)]
					\item $f_2(x_0)\geqslant\frac{1}{K}$.
					\item $f_2\leqslant0$, $f_2(x_0)=0$ and $f_1(x_0)\geqslant\frac{1}{K}$.
					\item $f_2\leqslant0$, $f_2(x_0)\leqslant-\frac{1}{K}$. Moreover, for any vertex $x\in V$, $f_2(x)\leqslant0$ and $f_1(x)\leqslant 0$, or $f_2(x)\leqslant-\frac{1}{K}$.
				\end{enumerate}
				Then there exists a constant $C$ depending only on $K, V, \omega$ such that the solution $u$ satisfies
				$$\|u\|_{l^\infty(G)}\leqslant C.$$
			\end{theorem}
			
			\begin{proof}
				
				In this proof, we use constants $c', c_1, c_2,c_3,c_4$ which depend only on $K, |V|,  \omega$. Suppose that in each case, there exist sequences $\{f_{2,m}\}$, $\{f_{1,m}\}$ and $\{c_{m}\}$, and corresponding solution sequence $\{u_m\}$, which is not uniformly bounded. By passing to subsequences, we can assume that
				$$\lim_{m\rightarrow+\infty}f_{2,m}=f_2,\ \lim_{m\rightarrow+\infty}f_{1,m}=f_1,\ \lim_{m\rightarrow+\infty}c_{m}=c,$$
				where $f_2,f_1,c$ satisfy the conditions in this theorem, and can assume that the vertex $x_0$ corresponding to different $f_{2,m}$, $f_{1,m}$ is the same one. By Theorem~\ref{tm:alternative}, we know that $u_m$ converges to $-\infty$ or $+\infty$ uniformly.
				
				Step 1: If $u_m$ converges to $-\infty$ uniformly, since
				we have
				$$0=\sum_{x\in V}-\Delta u_m(x)=c_m|V|+ \sum_{x\in V}\left(f_{2,m}e^{2u_m(x)}+f_{1,m}e^{u_m(x)}\right),$$
				take the limit with respect to $m$ on both sides and we obtain $c=0$, which is a contradiction.

				Step2: $u_m$ converges to $+\infty$ uniformly. Without loss of generality we may assume that $u_m(x)>0$ holds for every $m$ and $x\in V$. Then for any vertex $x$,
				\begin{equation}\label{ineq:main1}
					\begin{aligned}
						f_{2,m}(x) e^{2u_m(x)}+f_{1,m}(x)e^{u_m(x)}+c_m=\sum_{y\sim x}\omega_{xy}\left(u_m(x)-u_m(y)\right)\leqslant c'u_m(x),
					\end{aligned}
				\end{equation}
				where $c'=\max_{x\in V}\sum_{y\sim x}\omega_{xy}$. Then we show the contradiction for different cases.
				
				For case $(a)$, by \eqref{ineq:main1} we know that on the  vertex $x_0$,
				\begin{equation*}
					\begin{aligned}
						\frac{1}{K} e^{2u_m(x_0)}-Ke^{u_m(x_0)}-K\leqslant c'u_m(x_0).
					\end{aligned}
				\end{equation*}
				Let $m\rightarrow +\infty$ and the above inequality will no longer hold, which is a contradiction.

				For case $(b)$, on the vertex $x_0$, by the inequality \eqref{ineq:main1} we can derive that
				$$\frac{1}{K} e^{u_m(x_0)}-K\leqslant c' u_m(x_0).$$
				Let $m\rightarrow +\infty$ and we obtain a contradiction.
				
				Finally, we consider the case $(c)$. Note that 
				$$-\Delta u_m=F_2(x,u_m)\leqslant c_m\leqslant K$$
				on the set $\{f_{1,m}(x)\leqslant0\}$, and
				$$-\Delta u_m\leqslant -\frac{1}{K} e^{2u_m}+Ke^{u_m}+K\leqslant c_{3}$$
				on $\{f_{2,m}(x)\leqslant-\frac{1}{K}\}$. we know that $-\Delta u_m$ has a uniform upper bound and then $|\Delta u_m|$ has a uniform bound $c_4$. Consider some vertex $x_0$ satisfying $f_{2,m}(x_0)\leqslant-\frac{1}{K}$. Then we have
				$$-c_{4}\leqslant -\Delta u_m(x_0)\leqslant-\frac{1}{K} e^{2u_m(x_0)}+Ke^{u_m(x_0)}+K.$$
				Let $m\rightarrow +\infty$ and we obtain a contradiction.

			\end{proof}
			
			The above an a priori estimate is based on the condition $c\neq0$. For the case of $c=0$, the difficulty in proving similar an a priori estimate is how to establish a lower bound of the solution $u$. If we add a condition for $\overline{f_1}$, we can overcome this difficulty, and prove the following result.
			
			\begin{theorem}\label{tm2}
				Suppose that $c=0$, $u$ is a solution of \eqref{eq:maineq}, and there exists a constant $K>0$ satisfying
				$$\max_{x\in V} |f_2(x)|+\max_{x\in V} |f_1(x)|\leqslant K,$$
				$$|\overline{f_1}|\geqslant \frac{1}{K}.$$
				Assume that there exists some $x_0\in V$ such that one of the following conditions holds:
				\begin{enumerate}[(a)]
					\item $f_2(x_0)\geqslant\frac{1}{K}$.
					\item $f_2\leqslant0$, $f_2(x_0)=0$ and $f_1(x_0)\geqslant\frac{1}{K}$.
					\item $f_2\leqslant0$, $f_2(x_0)\leqslant-\frac{1}{K}$. Moreover, for any vertex $x\in V$, $f_2(x)\leqslant0$ and $f_1(x)\leqslant 0$, or $f_2(x)\leqslant-\frac{1}{K}$.
				\end{enumerate}
				Then there exists a constant $C'$ depending only on $K,V,\omega$ such that the solution $u$ satisfies
				$$\|u\|_{l^\infty(G)}\leqslant C'.$$
			\end{theorem}
			
			\begin{proof}
				
				Note that in the proof of Theorem~\ref{tm1}, step 2 is independent of the value of $c$. Thus, we only need to exclude case $(2)$ in Theorem~\ref{tm:alternative}. Suppose that there exist sequences $\{f_{2,m}\},\{f_{1,m}\}$, which satisfy the conditions of this theorem, and the corresponding solution sequence $\{u_m\}$ satisfying 
				$$\lim_{m\rightarrow+\infty}\max_{x\in V}u_m(x)=-\infty.$$
				
				For the case of $\overline{f_{1,m}}\geqslant K^{-1}$, 
				the following inequality holds:
				\begin{equation*}
					\begin{aligned}
						\Delta e^{u(x)}&=\sum_{y\sim x}\omega_{xy}\left(e^{u(y)}-e^{u(x)}\right)\\
						&\geqslant \sum_{y\sim x}\omega_{xy}e^{u(x)}\left(u(y)-u(x)\right)=e^{u(x)}\Delta u(x).\\                                                                                                             
					\end{aligned}
				\end{equation*}
				We derive that                         
				\begin{equation*}
					\begin{aligned}
						\Delta e^{-u_m} \geqslant -e^{-u_m}\Delta u_m=f_{2,m}e^{u_m}+f_{1,m}.
					\end{aligned}
				\end{equation*}
				By summing over $V$ and taking the limit with respect to $m$ on both sides, we obtain
				$$\overline{f_{1,m}}\leqslant0,$$
				which is a contradiction.
				
				For the case of $\overline{f_{1,m}}\leqslant -K^{-1}$.
				Let 
				$$h_m(x)=f_{2,m}(x)e^{u_m(x)}+f_{1,m}(x).$$
				For sufficiently large $m$, we have $\overline{h_m}<0$, and we can assume that $\overline{h_m}\leqslant-\frac{1}{2K}$. Thus the function $u_m$ is a solution to the equation
				$$-\Delta u=h_me^u.$$
				However, by a priori estimates for this type equation derived in \cite{sun2022brouwer}, the solution $u_m$ must be uniformly bounded, which is a contradiction.
				
			\end{proof}
			
			\begin{remark}
				In Theorem~\ref{tm1} and \ref{tm2}, uniform boundedness of the solutions may fail if any of the assumptions is removed.
				
			\end{remark}
			
			Here we need to mention why we assume that $\overline{f_{1}}\neq0$ in Theorem~\ref{tm2}, and we refer the readers to the following examples. In the following we assume that $V=\{x_1,x_2\}$, and $m\equiv1, \omega_{x_1,x_2}=1$. Suppose that $u$ is a solution to the equation
			$$-\Delta u=f_2e^{2u}+f_1e^u,$$
			and $u(x_1)=x,\ u(x_2)=y$.
			\begin{example}\label{ex3.4}
				For case $(a)$ in Theorem~\ref{tm2}, we let
				$$f_2(x_1)=1,\ f_2(x_2)=-2-\epsilon,\ f_1(x_1)=1,\ f_1(x_2)=-1,$$
				where $\epsilon>0$. Then the equation becomes
				\begin{equation*}
					\begin{aligned}
						\left\{
						\begin{aligned}
							x-y&=e^{2x}+e^x,\\
							y-x&=-(2+\epsilon)e^{2y}-e^y.
						\end{aligned}
						\right.
					\end{aligned}
				\end{equation*}
				Let $t=x-y>0$, then the equation is equivalent to 
				$$e^{2t}e^{2y}+e^te^y=e^{2x}+e^x=(2+\epsilon)e^{2y}+e^y=t.$$
				This is also equivalent to
				$$e^y=\frac{1-e^t}{e^{2t}-2-\epsilon}=\frac{-1+\sqrt{1+4(2+\epsilon )t}}{2(2+\epsilon)},$$
				which leads to
				$$\frac{1+\sqrt{1+4(2+\epsilon)t}}{2}=(1+\epsilon)\frac{t}{e^t-1}-t(1+e^t).$$
				One easily sees that the expression on the left-hand side is increasing with respect to $t$, while the right-hand side is decreasing. Note that the $t=0$ is the unique solution if $\epsilon=0$. Then for small $\epsilon>0$, we obtain a unique $t>0$, which yields the existence of $x,y$. However, when we let $\epsilon\rightarrow0$, $t$ will also tend to $0$, and $x,y\rightarrow-\infty$. Thus, we are unable to obtain the uniform a  priori estimates.
			\end{example}
			
			\begin{example}
				For case $(b)$ in Theorem~\ref{tm2}, we let
				$$f_2(x_1)=0,\ f_2(x_2)=-1-\epsilon,\ f_1(x_1)=1,\ f_1(x_2)=-1,$$
				where $\epsilon>0$. Then the equation becomes
				\begin{equation*}
					\begin{aligned}
						\left\{
						\begin{aligned}
							x-y&=e^x,\\
							y-x&=-(1+\epsilon) e^{2y}-e^y.
						\end{aligned}
						\right.
					\end{aligned}
				\end{equation*}
				Let $t=x-y>0$. Follow the method of Example ~\ref{ex3.4}, and we will derive
				$$1+\epsilon=\frac{(e^t-1)e^t}{t}.$$
				For small $\epsilon>0$, there exist unique $t$ and $x,y$. If we let $\epsilon\rightarrow0$, we will get that $t\rightarrow0$ and $x,y\rightarrow-\infty$.
			\end{example}
			
			\begin{example}
				For case $(c)$ in Theorem~\ref{tm2}, we let
				$$f_2(x_1)=0,\ f_2(x_2)=-1,\ f_1(x_1)=-1-\epsilon,\ f_1(x_2)=1+\epsilon,$$
				where $\epsilon>0$. Then the equation becomes
				\begin{equation*}
					\begin{aligned}
						\left\{
						\begin{aligned}
							x-y&=-(1+\epsilon)e^x,\\
							y-x&=- e^{2y}+(1+\epsilon)e^y.
						\end{aligned}
						\right.
					\end{aligned}
				\end{equation*}
				Let $t=y-x>0$. Following the above method, we obtain
				$$\left(1+\epsilon\right)^2=\frac{te^{2t}}{e^t-1}.$$
				Note that the right-hand side is increasing with respect to $t$. Thus, for small $\epsilon$ there admit unique $t$ and $x,y$. If we let $\epsilon\rightarrow0$, we will derive that $t\rightarrow0$ and $x,y\rightarrow-\infty$.
			\end{example}
			\section{Topological degree}\label{s4}
			
			To compute the value of $d_{f_2,f_1,c}$ on $G=(V,E,\omega,m)$, we list the following cases, which contain all possible situations:
			\begin{enumerate}[(a)]
				\item There exists some $x_0\in V$ satisfying $f_2(x_0)>0$.
				\item $f_2\leqslant0, f_1\leqslant0$.
				\item $f_2\leqslant0$, and there exists some $x_0\in V$ satisfying $f_2(x_0)=0$ and $f_1(x_0)>0$.
				\item $f_2\leqslant0$. $f_1(x)\leqslant 0$ when $f_2(x)=0$, and $f_1$ is positive somewhere.
			\end{enumerate}
			The above cases correspond to different conditions in a priori estimates, and in this section we compute the Brouwer degree for each case by using the homotopy invariance of the Brouwer degree and Lemma~\ref{lm3}. The homotopy invariance helps us reshape $f_2,f_1$ into simpler forms, and combining Lemma~\ref{lm3} with the homotopy invariance, we can reduce the graph to a two-vertex graph for simple computation.
			
			To apply the homotopy invariance, we consider the following equation for $t\in [0,1]$:
			$$-\Delta u=f_{2,t}e^{2u}+f_{1,t}e^u+c.$$
			Multiplying both sides of equation \eqref{eq:maineq} by $m(x)$, we can assume that $m(x)\equiv1$, and we always consider $m(x)\equiv1$ in this section.
			
			First, we consider $c\neq 0$. For the above four cases, we have the following theorem.
			
			\tm\label{tm3}
			Suppose that $u$ is a solution of \eqref{eq:maineq} and $c\neq 0$. Then we have
			\begin{equation*}
				\begin{aligned}
					d_{f_2,f_1,c}=\left\{
					\begin{aligned}
						1\ ,\ \ \   &c>0\ \text{in}\ (b)\ \text{and}\ (d),\\
						-1,\ \ \ &c<0\ \text{in}\ (a)\ \text{and}\ (c),\\
						0,\ \ \ &\text{otherwise}.\\
					\end{aligned}
					\right.
				\end{aligned}
			\end{equation*}
			\tmd
			
			\begin{proof}
				
				We compute the degree for each of the four cases separately. In this proof, we make different homotopy transformations in each case based on a priori estimates in Theorem~\ref{tm1}, which keep the Brouwer degree unchanged.
				
				For case $(a)$, let
				$$f_{2,t}(x)=\left\{
				\begin{aligned}
					&(1-t)f_2(x_0)+1, &\ x=x_0,\\
					&(1-t)f_2(x), &x\neq x_0,
				\end{aligned}
				\right.$$
				and 
				$$f_{1,t}(x)=(1-t)f_1(x).$$
				For $t=1$, applying Lemma~\ref{lm3}, we construct a new weighted connected graph $\widetilde{G}=(\widetilde{V},\widetilde{E},\widetilde{\omega},1)$ which contains two vertices, and the equation becomes
				\begin{equation*}
					\begin{aligned}
						\left\{
						\begin{aligned}
							&\widetilde{\omega}_{12}\left(u_1-u_2\right)=e^{2u_1}+\widetilde{c},\\
							&\widetilde{\omega}_{12}\left(u_2-u_1\right)=\widetilde{c},
						\end{aligned}
						\right.
					\end{aligned}
				\end{equation*}
				where $\widetilde{c}$ and $c$ have the same sign, satisfying $\widetilde{c}=\frac{|V|}{2}c$. One easily sees that the equations only have one solution for $c<0$ and have no solutions for $c>0$. By computing that
				$$\det(DH_{\widetilde{f_2},\widetilde{f_1},\widetilde{c}}(u))=-2\widetilde{\omega}_{12}e^{u_1}<0,$$
				where $\widetilde{f_2}=(1,0),\widetilde{f_1}=0$, we derive that 
				\begin{equation*}
					\begin{aligned}
						d_{f_2,f_1,c}=d_{f_{2,1},f_{1,1},c}=d_{\widetilde{f_2},\widetilde{f_1},\widetilde{c}}=\left\{
						\begin{aligned}
							-1 &,   &c<0,\\
							0&, &c>0.\\
						\end{aligned}
						\right.
					\end{aligned}
				\end{equation*}
				
				For case $(b)$, 
				there exists $f_2(x_0)<0$, and we make the homotopy transformation
				$$f_{2,t}(x)=\left\{
				\begin{aligned}
					&(1-t)f_2(x_0)-t, &\ x=x_0,\\
					&(1-t)f_2(x), &x\neq x_0,
				\end{aligned}
				\right.$$
				and 
				$$f_{1,t}(x)=(1-t)f_1(x).$$
				For $t=1$, we apply Lemma~\ref{lm3} and derive the following equations on $\widetilde{G}$:
				\begin{equation*}
					\begin{aligned}
						\left\{
						\begin{aligned}
							&\widetilde{\omega}_{12}\left(u_1-u_2\right)=-e^{2u_1}+\widetilde{c},\\
							&\widetilde{\omega}_{12}\left(u_2-u_1\right)=\widetilde{c}.
						\end{aligned}
						\right.
					\end{aligned}
				\end{equation*}
				It implies that
				\begin{equation*}
					\begin{aligned}
						d_{f_2,f_1,c}=\left\{
						\begin{aligned}
							1 &,   &c>0,\\
							0&, &c<0.\\
						\end{aligned}
						\right.
					\end{aligned}
				\end{equation*}
				
				For case $(c)$, we let
				$$f_{2,t}(x)=(1-t)f_2(x)$$
				and
				$$f_{1,t}(x)=\left\{
				\begin{aligned}
					&(1-t)f_1(x_0)+t, &\ x=x_0,\\
					&(1-t)f_1(x), &x\neq x_0.
				\end{aligned}
				\right.$$
				By Lemma~\ref{lm3}, this derives the equations on $\widetilde{G}$:
				\begin{equation*}
					\begin{aligned}
						\left\{
						\begin{aligned}
							&\widetilde{\omega}_{12}\left(u_1-u_2\right)=e^{u_1}+\widetilde{c},\\
							&\widetilde{\omega}_{12}\left(u_2-u_1\right)=\widetilde{c},
						\end{aligned}
						\right.
					\end{aligned}
				\end{equation*}
				and one easily sees that
				\begin{equation*}
					\begin{aligned}
						d_{f_2,f_1,c}=\left\{
						\begin{aligned}
							-1 &,   &c<0,\\
							0&, &c>0.\\
						\end{aligned}
						\right.
					\end{aligned}
				\end{equation*}
				
				For case $(d)$, we first use the homotopy transformation
				$$f_{2,t}(x)=\left\{
				\begin{aligned}
					&(1-t)f_2(x)-t, &\ \text{if}\ f_2(x)<0,\\
					&(1-t)f_2(x), &\ \text{if}\ f_2(x)=0,
				\end{aligned}
				\right.$$
				and 
				$$f_{1,t}(x)=(1-t)f_1(x).$$
				Note that when $t=1$, we return to case $(b)$. Thus using the above result for case $(b)$, we get
				\begin{equation*}
					\begin{aligned}
						d_{f_2,f_1,c}=\left\{
						\begin{aligned}
							1 &,   &c>0,\\
							0&, &c<0.\\
						\end{aligned}
						\right.
					\end{aligned}
				\end{equation*}

			\end{proof}
			
			If we consider $c=0$, by a priori estimates for $\overline{f_1}\neq0$ we have proved in Theorem~\ref{tm2}, we can also compute the Brouwer degree as follows.
			
			\tm\label{tm4}
			Suppose that $u$ is a solution to \eqref{eq:maineq} and $\overline{f_1}\neq 0,\  c= 0$. Then we have
			\begin{equation*}
				\begin{aligned}
					d_{f_2,f_1,c}=\left\{
					\begin{aligned}
						1 &, \  \ \overline{f_1}>0 \ \text{in}\ (d),\\
						-1&, \ \  \overline{f_1}<0 \ \text{in}\ (a)\ \text{and}\ (c),\\
						0&, \ \ \text{otherwise}.
					\end{aligned}
					\right.
				\end{aligned}
			\end{equation*}
			\tmd
			
			\begin{proof}
				
				For case $(a)$, if $\overline{f_1}>0$, we let
				$$f_{2,t}(x)=\left\{
				\begin{aligned}
					&(1-t)f_2(x)+t, &\ x=x_0,\\
					&(1-t)f_2(x), &\ x\neq x_0,
				\end{aligned}
				\right.$$
				and
				$$f_{1,t}(x)=\left\{
				\begin{aligned}
					&(1-t)f_1(x)+t, &\ x=x_0,\\
					&(1-t)f_1(x), &\ x\neq x_0.
				\end{aligned}
				\right.$$
				Note that 
				$$\overline{f_{1,t}}=(1-t)\overline{f_1}+t\geqslant\min\{1,\overline{f_1}\}>0.$$
				By Theorem~\ref{tm2} we know this homotopy transformation doesn't change the value of the Brouwer degree. Thus by Lemma~\ref{lm3}, we reduce the graph $G$ to a two-vertex graph. Let $t=1$ and we get the equations
				\begin{equation*}
					\begin{aligned}
						\left\{
						\begin{aligned}
							&\widetilde{\omega}_{12}\left(u_1-u_2\right)=e^{2u_1}+e^{u_1},\\
							&\widetilde{\omega}_{12}\left(u_2-u_1\right)=0.
						\end{aligned}
						\right.
					\end{aligned}
				\end{equation*}
				We easily find that the equations have no solutions. This implies that 
				$$d_{f_2,f_1,c}=0.$$
				If $\overline{f_1}<0$, we apply the same homotopy transformation to $f_2$, and let
				$$f_{1,t}(x)=\left\{
				\begin{aligned}
					&(1-t)f_1(x)-t, &\ x=x_0,\\
					&(1-t)f_1(x), &\ x\neq x_0.
				\end{aligned}
				\right.$$
				Then setting $t=1$ and applying Lemma~\ref{lm3}, we obtain the following equations on the reduced graph
				\begin{equation*}
					\begin{aligned}
						\left\{
						\begin{aligned}
							&\widetilde{\omega}_{12}\left(u_1-u_2\right)=e^{2u_1}-e^{u_1},\\
							&\widetilde{\omega}_{12}\left(u_2-u_1\right)=0.
						\end{aligned}
						\right.
					\end{aligned}
				\end{equation*}
				It is not difficult to see that there exists a unique  solution $u=(u_1,u_2)=(0,0)$, and 
				$$d_{f_2,f_1,c}=\sgn \det \left(
				\begin{matrix}
					\widetilde{\omega}_{12}-2e^{2u_1}+e^{u_1} & -\widetilde{\omega}_{12}\\
					-\widetilde{\omega}_{12} & \widetilde{\omega}_{12}
				\end{matrix}
				\right)=\sgn\left(-\widetilde{\omega}_{12}\right)=-1.$$
				
				For case $(b)$, summing both sides of the equation, we obtain
				$$0=\sum_{ V}\left(f_2e^{2u}+f_1e^u\right)<0,$$
				which is a contradiction. Thus the equation has no solutions, and
				$$d_{f_2,f_1,c}=0.$$
				
				For case $(c)$, we consider $\overline{f_1}>0$ first. Let
				$$f_{2,t}(x)=(1-t)f_2(x)$$
				and
				$$f_{1,t}(x)=\left\{
				\begin{aligned}
					&(1-t)f_1(x)+t, &\ x=x_0,\\
					&(1-t)f_1(x), &\ x\neq x_0.
				\end{aligned}
				\right.$$
				Note that $f_1\geqslant0, f_{1,1}\geqslant0$ and $f_1\neq0$. Similar to the argument for case $(b)$, we have 
				$$d_{f_2,f_1,c}=0.$$
				For $\overline{f_1}<0$, we choose a vertex $x_1\in V$ which is not the same vertex as $x_0$. We transform $f_{2,t}$ as above and 
				$$f_{1,t}(x)=\left\{
				\begin{aligned}
					&(1-t)f_1(x)+t, &\ x=x_0,\\
					&(1-t)f_1(x)-et, &\ x=x_1,\\
					&(1-t)f_1(x), &\ x\neq x_0,x_1.
				\end{aligned}
				\right.$$
				We note that $\overline{f_{1,t}}<0$ and has a negative upper bound. Thus applying Lemma~\ref{lm3}, for $t=1$ we derive the equations
				\begin{equation*}
					\begin{aligned}
						\left\{
						\begin{aligned}
							&\widetilde{\omega}_{12}\left(u_1-u_2\right)=e^{u_1},\\
							&\widetilde{\omega}_{12}\left(u_2-u_1\right)=-e^{u_2+1}.
						\end{aligned}
						\right.
					\end{aligned}
				\end{equation*}
				The unique solution is $u=(u_1,u_2)=(\ln\widetilde{\omega}_{12},\ln\widetilde{\omega}_{12}-1)$, and we have
				$$d_{f_2,f_1,c}=-1.$$
				
				For case $(d)$, if $\overline{f_1}<0$, we transform that
				$$f_{2,t}(x)=\left\{
				\begin{aligned}
					&(1-t)f_2(x)-t, &\ \text{if}\ f_2(x)<0,\\
					&(1-t)f_2(x), &\ \text{if}\ f_2(x)=0,
				\end{aligned}
				\right.$$
				and 
				$$f_{1,t}(x)=\left\{
				\begin{aligned}
					&(1-t)f_1(x)-t, &\ \text{if}\ f_2(x)<0,\\
					&(1-t)f_1(x), &\ \text{if}\ f_2(x)=0.
				\end{aligned}
				\right.$$
				Applying Lemma~\ref{lm3} we remove the vertices where $f_2=f_1=0$, and on a new weighted connected finite graph, we have the equation for $t=1$
				$$-\Delta u=-e^{2u}-e^u<0.$$
				Summing both sides, and we know that the equation has no solutions. Thus
				$$d_{f_2,f_1,c}=0.$$
				If $\overline{f_1} >0$, we transform $f_{2,t}$ as above, and
				$$f_{1,t}(x)=\left\{
				\begin{aligned}
					&(1-t)f_1(x)+t, &\ \text{if}\ f_2(x)<0,\\
					&(1-t)f_1(x), &\ \text{if}\ f_2(x)=0.
				\end{aligned}
				\right.$$
				Repeat the above process, and we get the equation
				$$\Delta u=e^u(e^u-1)$$
				on a new weighted connected finite graph $\widetilde{G}=(\widetilde{V},\widetilde{E},\widetilde{\omega},1)$. For any solution $u$ to this equation, suppose that 
				$$\max_{x\in \widetilde{V}}u(x)=u(y)>0.$$
				This implies $\Delta u(y)>0$, which means there exists a vertex $y_1\in \widetilde{V}$ satisfying $u(y_1)>u(y)$, and this is a contradiction. Repeat the same argument for $\min_{x\in \widetilde{V}}u(x)$, we derive that $u\equiv0$, and it is the unique solution to this equation. To compute the Brouwer degree, note that
				$$d_{f_2,f_1,c}=\sgn \det \left(-\Delta+I\right).$$
				It is known that $-\Delta$ is semi-definite. Thus
				$$d_{f_2,f_1,c}=1.$$

			\end{proof}
			
			\
			
			\section{Existence results}\label{se5}
			In this section we prove the existence results for equation \eqref{eq:maineq} when $n=2$. It is known that there exists at least one solution if the Brouwer degree is not equal to zero. Thus in this section our main aim is to study the existence of solutions to the equation when the Brouwer degree is equal to zero. 
			
			\subsection{Simple cases}\label{sub5.1}
			
			By the Kronecker existence theorem and the results in Section~\ref{s4}, one easily sees the existence results in the following theorems.
			
			\tm\label{tm5}\label{5.1}
			If $c>0$ in cases $(b)$ and $(d)$, or $c<0$ in cases $(a)$ and $(c)$,  equation \eqref{eq:maineq} must have at least one solution.
			\tmd
			
			\tm\label{tm6}
			When $c=0$, if we suppose that $\overline{f_1}>0$ in case $(d)$, or $\overline{f_1}<0$ in cases $(a)$ and $(c)$, equation \eqref{eq:maineq} must have at least one solution.
			\tmd
			For the simple case $(b)$, we get the following result.
			\tm\label{tm7}
			Equation \eqref{eq:maineq} has no solutions in case $(b)$ if $c\leqslant0$. 
			\tmd

			In the rest of this section, we only consider the remaining cases.		    
			\begin{enumerate}[(A)]
				\item $c\geqslant0$ in cases $(a)$ and $(c)$.
				\item $c\leqslant0$ in case $(d)$.
			\end{enumerate}		   
			Note that the Brouwer degree is equal to zero, and we need to go back to equation \eqref{eq:maineq} itself to see if there exists a solution. Suppose that every critical point of $J_{f_2,f_1,c}$ is nondegenerate, then we deduce that the equation \eqref{eq:maineq} has no solutions or at least two different solutions. Otherwise,  equation \eqref{eq:maineq} may have only one solution.

			\subsection{Some cases for $\mathbf{c\neq0}$}\label{5.2}
			
			We discuss the existence results for $c\neq0$
			in this subsection under certain conditions for $\overline{f_1}$. The main method is sub- and supersolutions in Lemma ~\ref{lm2}, and we prove that equation \eqref{eq:maineq} has at least one solution within a certain range of $c$, and no solutions outside this range.
			
			For case $(A)$, if there exists a solution $u_1$ for $c=c_1>0$, for any $0<c_2<c_1$, we note that
			$$-\Delta u_1=f_2 e^{2u_1}+f_1e^{u_1}+c_1>f_2 e^{2u_1}+f_1e^{u_1}+c_2.$$
			This shows that $u_1$ is a supersolution to the equation
			\begin{equation}\label{eq:c2}
				\begin{aligned}
					-\Delta u=f_2e^{2u}+f_1e^{u}+c_2.
				\end{aligned}
			\end{equation}
			By choosing $u_2=-M$, where $M$ is a sufficiently large positive constant, we find that $u_2$ can be a subsolution to equation \eqref{eq:c2}, and we can assume that $u_2<u_1$. Consider the minimizing problem
			$$\inf_{u_2\leqslant u\leqslant u_1}J_{f_2,f_1,c_2}(u).$$
			Let $\{v_i\}_{i=1}^{+\infty}$ be a minimizing sequence which satisfy $u_1\leqslant v_i\leqslant u_2$ and
			$$\lim_{i\rightarrow+\infty}J_{f_2,f_1,c_2}(v_i)=\inf_{u_1\leqslant u\leqslant u_2}J_{f_2,f_1,c_2}(u).$$
			Since we discuss on a finite graph, it is not difficult to choose a subsequence converging to a function $v$, which is a minimizer of $J_{f_2,f_1,c_2}$ in $\{u:V\rightarrow\mathbb{R}:u_1\leqslant u\leqslant u_2\}$, and is a solution to \eqref{eq:c2} by Lemma~\ref{lm2}. This yields that equation \eqref{eq:maineq} has a solution for any constant $0<c<c_1$.
			
			The above discussion also applies to case $(B)$ by choosing $u_2$ a sufficiently large positive constant as a supersolution. Thus, we derive that equation \eqref{eq:maineq} has a solution when $c\in (0,c_{f_2,f_1})$ in case $(A)$, and $c\in (-c_{f_2,f_1},0)$ in case $(B)$, where $c_{f_2,f_1}$ may be zero, a positive constant, or infty, depending only on $f_2,f_1,G$. In the following part of this subsection, by adding some restrictions on $\overline{f_1}$, we prove that $c_{f_2,f_1}$ cannot be zero. In Subsection~\ref{5.4}, we will prove that $c_{f_2,f_1}$ is a positive constant, and if removing the restrictions on $f_1$, we can also prove the same results as long as we make some reasonable assumptions about $f_2$. 
			
			In case $(A)$, we add an additional condition that $\overline{f_1}<0$, and suppose $\overline{f_1}>0$ in case $(B)$. This yields the existence of at least one solution when $c$ is sufficiently close to zero.
			
			\tm\label{tm8}
			Suppose that $\overline{f_1}<0$ in case $(A)$, and $\overline{f_1}>0$ in case $(B)$. Then equation \eqref{eq:maineq} has at least one solution when $c$ is sufficiently close to zero. 
			\tmd
			
			\begin{proof}
				The proof is based on the method of sub- and supersolutions. By the above analysis on this subsection, we just need to construct a suitable subsolution and supersolution.
				
				Case $(A)$: We have already obtained a subsolution $u_2=-M$ where $M$ is a sufficiently large positive constant. To construct a supersolution, we consider the constants $a>0$ and $b=\ln a$, and the function $v:V\rightarrow \mathbb{R}$ which satisfies the equation
				\begin{equation}\label{eq:v}
					\begin{aligned}
						\Delta v=a(f_2-\overline{f_2})+f_1-\overline{f_1}.
					\end{aligned}
				\end{equation}			 
				We choose the value of $a$ later. It is known that 
				$$\dim\ker \Delta=1.$$
				Thus by restricting $v$ to be zero at a certain vertex, the function $v$ is the unique solution of \eqref{eq:v}. We limit the range of values of $a$ to $(0,1]$, and thus $v$ is uniformly bounded.  Note that
				\begin{equation*}
					\begin{aligned}
						&-\Delta(av+b)=a^2(f_2-\overline{f_2})+a(f_1-\overline{f_1})\\
						=&f_2 e^{2(av+b)}-a^2f_2 \left(e^{2av}-1\right)-a^2\overline{f_2}
						+f_1e^{av+b}
						-af_1\left(e^{av}-1\right)-a\overline{f_1}\\
						=&f_2 e^{2(av+b)}+f_1e^{av+b}-a\left[af_2\left(e^{2av}-1\right)+a\overline{f_2}+f_1\left(e^{av}-1\right)+\overline{f_1}\right].
					\end{aligned}
				\end{equation*}
				Let $a$ be sufficiently small satisfying
				$$\left|af_2\left(e^{2av}-1\right)+a\overline{f_2}+f_1\left(e^{av}-1\right)\right|\leqslant -\frac{1}{2}\overline{f_1}.$$
				This yields that
				$$-\Delta(av+b)\geqslant f_2 e^{2(av+b)}+f_1 e^{av+b}-\frac{1}{2}a\overline{f_1}.$$
				Thus for $0<c<-\frac{1}{2}a\overline{f_1}$, equation \eqref{eq:maineq} has a supersolution $av+b$, and then has at least one solution.
				
				Case $(B)$: Note that $u_2=M>0$ sufficiently large is a supersolution of \eqref{eq:maineq}. By repeating the same process in Case $(A)$, we derive that
				$$-\Delta(av+b)\leqslant f_2 e^{2(av+b)}+f_1e^{av+b}-\frac{1}{2}a\overline{f_1}.$$
				For $-\frac{1}{2}a\overline{f_1}<c<0$, it is not difficult to see that $av+b$ is a subsolution.
			\end{proof}
			
			\subsection{$\mathbf{c=0}$}
			In this subsection we concentrate on the equation
			\begin{equation}\label{eq:c=0}
				\begin{aligned}
					-\Delta u=f_2 e^{2u}+f_1e^u.
				\end{aligned}
			\end{equation}			
			The unsolved cases are that $\overline{f_1}\geqslant0$ in cases $(a)$ and $(c)$, and $\overline{f_1}\leqslant0$ in case $(d)$. The existence results actually depend on both $f_2$ and $f_1$. In the following theorem, we give some sufficient conditions for equation \eqref{eq:maineq} to have a solution.
			
			\tm\label{tm9}
			We suppose that $\overline{f_1}>0$ in cases $(a)$ and $(c)$, and $\overline{f_1}<0$ in case $(d)$. Then for each case, there exists a function $f_2^*(\max_{x\in V} |f_1(x)|, \overline{f_1})$ depending only on $\max_{x\in V} |f_1(x)|, \overline{f_1}$ and $G$, and we derive that
			\begin{enumerate}[(a)]
				\item Equation \eqref{eq:c=0} has a solution when $f_2\leqslant f_2^*$ in case $(a)$, where $f_2^*$ changes sign.
				\item Equation \eqref{eq:c=0} has a solution when $f_2\leqslant f_2^*\leqslant0$ in case $(c)$, where $\overline{f_2^*}<-C(\max_{x\in V} |f_1(x)|)\overline{f_1}$.
				\item Equation \eqref{eq:c=0} has a solution when $0\geqslant f_2\geqslant f_2^*$ in case $(d)$.			  	
			\end{enumerate}
			\tmd
			
			\begin{proof}
				Assume that 
				$$-K\leqslant f_1\leqslant K.$$
				Consider the equation 
				\begin{equation}\label{eq:H}
					\begin{aligned}
						-\Delta u=He^u,
					\end{aligned}
				\end{equation}
				where the function $H$ changes sign and satisfies 
				$$\overline{H}<0,\ -tK\leqslant H\leqslant tK,$$
				and the constant $t>1$ we will choose later. Using the results in \cite{sun2022brouwer}, we know a priori estimates for solutions to equation \eqref{eq:H} that 
				$|u|\leqslant C(tK)$, and equation \eqref{eq:H} must have at least one solution. 
				
				For case $(a)$, we consider the equation 
				$$-\Delta u_k=f_2e^{2u_k}+f_1e^{u_k}-c_k$$
				and the corresponding solution $u_k$, where $c_k>0$ and  
				$$\lim_{k\rightarrow\infty}c_k=0.$$
				The existence of $u_k$ is based on Theorem~\ref{5.1}. If $\{u_k\}$ is uniformly bounded, by choosing a weak convergent subsequence, we can find a solution to \eqref{eq:c=0}. Otherwise we can assume that 
				$$\lim_{k\rightarrow\infty}\|u_k\|_{l^{\infty}}=+\infty.$$
				By the proof of Theorem~\ref{tm1}, we know that it will lead to a contradiction if 
				$$\lim_{k\rightarrow\infty}\max_{x\in V} u_k(x)=+\infty.$$
				Then by Theorem~\ref{tm:alternative}, for any vertex $x\in V$, we have
				$$\lim_{n\rightarrow\infty} u_k(x)=-\infty.$$
				Thus we can choose some $c_k>0$ and sufficiently negative $u_k$ as a subsolution to equation \eqref{eq:c=0}. We only need to construct a supersolution.
				
				Note that
				\begin{equation*}
					\begin{aligned}
						-\Delta u=He^u= (H-f_1)e^{-u}e^{2u}+f_1e^u,
					\end{aligned}
				\end{equation*}
				and recall that $|u|\leqslant C(tK)$. We consider a sign-changing function $f_2^*$ which satisfies
				$f_2^*\leqslant (H-f_1)e^{-v}$, where $v$ is a solution to \eqref{eq:H}. We can claim the existence of $f_2^*$. For example, we assume that $f_1$ attains its maximum value at the vertex $x_0$, and we can choose 
				\begin{equation*}
					\begin{aligned}
						H(x)=\left\{
						\begin{aligned}
							& f_1+\frac{\max_{x\in V} f_1(x)}{2m(x)},\ x=x_0, \\
							& f_1-\frac{|V|\overline{f_1}+\max_{x\in V} f_1(x)}{\left(|V|-1\right)m(x)}, \ x\neq x_0.\\
						\end{aligned}
						\right.
					\end{aligned}
				\end{equation*}
				Since  
				\begin{equation*}
					\begin{aligned}
						\sum_{V}H&=\sum_{V}f_1+\frac{1}{2}\max_{x\in V} f_1(x)-\frac{|V|\overline{f_1}+\max_{x\in V} f_1(x)}{\left(|V|-1\right)}\cdot\left(|V|-1\right)\\
						&=-\frac{1}{2}\max_{x\in V} f_1(x)<0,   
					\end{aligned}
				\end{equation*}				
				we know $H$ changes sign and
				equation \eqref{eq:H} has a solution $v$. We assume that $m=\min_{x\in V} m(x)>0$, and clearly 
				$$|H(x)|\leqslant K\left(1+\frac{3}{m}\right).$$
				Thus we can choose $t=1+\frac{3}{m}$. Let
				\begin{equation*}
					\begin{aligned}
						f_2^*(x)=\left\{
						\begin{aligned}
							& \left(H-f_1\right)e^{-C(tK)},\ x=x_0, \\
							& \left(H-f_1\right)e^{C(tK)}, \ x\neq x_0.\\
						\end{aligned}
						\right.
					\end{aligned}
				\end{equation*}
				One easily sees that $f_2^*$ changes sign and 
				$$\overline{f_2^*}\leqslant \overline{H-f_1}<0.$$
				Thus for $f_2\leqslant f_2^*$, we have
				$$-\Delta v\geqslant f_2^*e^{2v}+f_1e^v\geqslant f_2e^{2v}+f_1e^v,$$
				and $v$ is a supersolution to \eqref{eq:c=0}. This yields the existence of solutions to \eqref{eq:c=0}.
				
				For case $(c)$, we repeat the above process. Let $t>3$, $H=f_1-g$, where $0\leqslant g\leqslant (t-1)\max_{x\in V} f_1(x)$ and $\overline{g}>\overline{f_1}$. Since in case $(c)$ there exists a vertex $x_0$ satisfying $f_2(x_0)=0$ and $f_1(x_0)>0$, we can reasonably assume that $g(x_0)=0$, which makes $H$ sign-changing. Note that for any $t>3$, there must exist such $g$. Under these conditions, equation \eqref{eq:H} has a solution $v$ satisfying $|v|\leqslant C(tK)$. Let
				$$f_2^*=-ge^{C(tK)},$$
				and for any $f_2\leqslant f_2^*$, $v$ is a supersolution to \eqref{eq:c=0}, which leads to a solution of \eqref{eq:c=0}. For a fixed $t>3$, let $K=\max_{x\in V} |f_1(x)|$. Then $f_2^*$ should satisfy the necessary condition
				$$ \overline{f_2^*}<-C(\max_{x\in V} |f_1(x)|)\overline{f_1}<0.$$
				
				For case $(d)$, consider $u_1=M$ a sufficiently large constant as a supersolution to \eqref{eq:c=0}. To find a subsolution, we apply the above method, and only need to construct $H,f_2^*$ such that $H$ changes sign and 
				$$ \left(H-f_1\right)e^{-v}\leqslant f_2^*\leqslant0,$$
				where $v$ is a solution to \eqref{eq:H}. We can let
				$$ H=f_1-\frac{1}{2}\max_{x\in V} f_1(x),\ t=\frac{3}{2},\ f_2^*=\left(H-f_1\right)e^{-C(\frac{3}{2}\max_{x\in V} |f_1(x)|)}\leqslant 0.$$ 
				Thus for any $f_2^*\leqslant f_2\leqslant0$ in case $(d)$, $v$ is a subsolution to \eqref{eq:c=0}, and equation \eqref{eq:c=0} must have at least one solution.
			\end{proof}
			
			\begin{remark}\label{re1}
				In the proof of Theorem~\ref{tm9}, there are many ways to choose $f_2^*$. However, we should indicate that these restrictions on $f_2$ are necessary in a certain sense. Removing these restrictions might make the equation unsolvable.
			\end{remark}
			
			We list some examples to support Remark~\ref{re1}.
			
			\begin{example}\label{ex1}
				For $\overline{f_1}>0$ in case $(a)$, suppose that $f_2\geqslant0$. Then equation \eqref{eq:c=0} has no solutions.

			\end{example}
			\begin{proof}
				Assume that there exists a solution $u$ of \eqref{eq:c=0}. Note that 
				$$\sum_{V}\left(f_2e^u+f_1\right)\geqslant \sum_{V}f_1>0.$$
				By the existence results in \cite{sun2022brouwer}, we deduce that equation \eqref{eq:c=0} has no solutions, which leads to a contradiction.
			\end{proof}
			\begin{example}\label{ex2}
				For $\overline{f_1}>0$ in case $(c)$, suppose that 
				$$\max_{x\in V} |f_2(x)|+\max_{x\in V} |f_1(x)|\leqslant K,\ 
				|\overline{f_1}|\geqslant K^{-1},$$
				and fix $f_1$. Then condition $(b)$ in Theorem~\ref{tm2} holds. Thus for any $\overline{f_2}>-e^{-C'}\overline{f_1}$ satisfying case $(c)$, we have
				$$\overline{f_2e^u+f_1}>0,$$
				where $C'$ is the constant in Theorem~\ref{tm2}, which results in equation \eqref{eq:c=0} being unsolvable.
			\end{example}
			
			\begin{example}\label{ex3}
				For $\overline{f_1}<0$ in case $(d)$, let $G=(V,E,\omega,m)$ be a graph with only two vertices $x_1,x_2$, and $m\equiv1$, $\omega_{x_1x_2}=1$. We denote by $u(x_1)=x$, $ u(x_2)=y$, and  suppose that 
				$$f_2(x_1)=0,\ f_2(x_2)=-K,\ f_1(x_1)=-a,\ f_1(x_2)=b,$$
				where $a>b>0$, $K>0$. Then we need to solve
				\begin{equation*}
					\begin{aligned}
						\left\{
						\begin{aligned}
							& x-y=-ae^x, \\
							& y-x=-Ke^{2y}+be^y.\\
						\end{aligned}
						\right.
					\end{aligned}
				\end{equation*}
				Denote by $t=y-x>0$, and the following holds:
				$$ae^x=be^y-Ke^{2y}=be^te^x-Ke^{2t}e^{2x}.$$
				This yields that
				$$\frac{K}{a}t=Ke^x=\dfrac{be^t-a}{e^{2t}}.$$
				However, if $K$ is sufficiently large, the equation about $t>0$ is unsolvable since $a>b$.
				
				In addition, for any $f_2$ satisfying $f_2(x_1)\leqslant0,\ f_2(x_2)\leqslant-K$, by the method of sub- and supersolutions, we easily obtain that equation \eqref{eq:c=0} has no solutions; otherwise, the solution would serve as a subsolution for the above example.
			\end{example}
			
			For $\overline{f_1}=0$, we also have similar results.
			\begin{corollary}\label{co5.10}
				For $\overline{f_1}=0$ in cases $(a),(c),(d)$, there exists a function $f_2^*(f_1)$ depending on $f_1$ and $G$ for each case, and equation \eqref{eq:c=0} is solvable when $f_2\leqslant f_2^*$ in cases $(a),(c)$, $0\geqslant f_2\geqslant f_2^*$ in case $(d)$. In particular, when $f_1=0$, equation \eqref{eq:c=0} has a solution only if $f_2$ changes sign and $\overline{f_2}<0$.
			\end{corollary}
			
			\begin{proof}
				To derive the existence of solutions, we use the method of sub- and supersolutions.
				
				The situation of $f_1=0$ only occurs in case $(a)$. By the results in \cite{sun2022brouwer} we easily obtain that the solutions exist only if  $f_2$ changes sign and $\overline{f_2}<0$. Thus we consider $f_1\neq0$ in the following proof.
				
				For cases $(a),(c)$, by the proof of Theorem~\ref{tm9}, we only need to  construct a supersolution. Let 
				$$g=\left\{
				\begin{aligned}
					& f_1+\epsilon, \ f_1\neq0,\\
					& f_1,\ f_1=0,\\
				\end{aligned}
				\right.$$
				where $\epsilon$ is a small constant satisfying $0<\epsilon <\min_{f_1\neq 0}|f_1(x)|$. The new function $g$ still belongs to the original case ($(a)$ or $(c)$). Applying Theorem~\ref{tm9}, we obtain that there exists a function $f_2^*$ and the equation 
				$$-\Delta v=f_2e^{2v}+ge^v$$
				is solvable for any $f_2\leqslant f_2^*$. Thus the solution $v$ can serve as a supersolution of \eqref{eq:c=0}.
				
				For case $(d)$, we need to construct a subsolution by the proof of Theorem~\ref{tm9}. Letting 
				$$g=\left\{
				\begin{aligned}
					& f_1-\epsilon, \ f_1\neq0,\\
					& f_1,\ f_1=0,\\
				\end{aligned}
				\right.$$
				and repeating the above discussion, we get the existence result. 
			\end{proof}
			\subsection{The complete discussion for existence results}\label{5.4}
			
			In this subsection, we discuss the remaining cases in subsection~\ref{5.2} and the value estimates of $c_{f_2,f_1}$.  Recall that in Theorem~\ref{tm8}, we only consider $\overline{f_1}<0$ in case $(A)$ and  $\overline{f_1}>0$ in case $(B)$. If removing the restrictions on $\overline{f_1}$ and adding some reasonable conditions on $f_2$, using Theorem~\ref{tm9} and Corollary~\ref{co5.10}, we can obtain the existence of solutions when $c$ is sufficiently close to zero. 
			
			\tm\label{tm10}
			Suppose that $\overline{f_1}\geqslant0, c>0$ in cases $(a)$ and $(c)$, and $\overline{f_1}\leqslant0, c<0$ in case $(d)$, where $f_1$ is not always zero. For each case there exists a function $\widetilde{f_2}(f_1)\neq0$ depending only on $f_1$ and $G$, which makes equation \eqref{eq:maineq} solvable for $c$ sufficiently close to zero if assuming $f_2\leqslant \widetilde{f_2}(f_1)$ in cases $(a)$ and $(c)$, and $f_2\geqslant \widetilde{f_2}(f_1)$ in case $(d)$.
			\tmd
			
			\begin{proof}
				Let $\epsilon$ be a constant satisfying
				$$0<\epsilon<\max_{x\in V}f_1(x).$$
				For cases $(a)$ and $(c)$, consider a new function
				$$f_1'=f_1+\epsilon,$$
				which also belongs to the original case ($(a)$ or $(c)$). By Theorem~\ref{tm9}, we derive that there exists $\widetilde{f_2}(f_1)$ such that for $f_2\leqslant \widetilde{f_2}$, the equation
				$$-\Delta u=f_2e^{2u}+f_1'e^u$$
				has a solution $v$. Note that
				$$-\Delta v=f_2e^{2v}+f_1e^v+\epsilon e^v.$$
				This implies that for any constant $0<c<\min_{x\in V}\epsilon e^v$, $v$ is a supersolution to equation \eqref{eq:maineq}. Recall that $u=-M$ sufficiently negative is a subsolution. Let us choose $-M<v$, and we can solve equation \eqref{eq:maineq} by the method of sub- and supersolutions.
				
				For case $(d)$,  we consider
				$$f_1'=f_1-\epsilon,$$
				which also belongs to case $(d)$. By similar discussion and noting that $u=M$ sufficiently large is a supersolution, we easily obtain the existence of solutions to equation \eqref{eq:maineq} if $c$ is close to zero.
			\end{proof}
			
			\begin{remark}\label{re2}
				In particular, if $f_1\equiv0$ in Theorem~\ref{tm10}, which only occurs in case $(a)$, according to the results in \cite{sun2022brouwer}, equation \eqref{eq:maineq} is solvable if and only if $\overline{f_2}<0<\max_{x\in V}f_2(x)$.
			\end{remark}
			
			Combining Theorem~\ref{tm8} and Theorem~\ref{tm10}, we obtain that $c_{f_2,f_1}$ cannot be zero if $f_2,f_1,c$ satisfy the conditions in one of these two theorems. In the following part, we denote by $f_{1,-}(x)=\max\{-f_1(x),0\}$ and $f_{2,-}$ in the same way. We claim that under these conditions, $c_{f_2,f_1}$ must be a positive constant, and we give the estimates of $c_{f_2,f_1}$.
			
			\tm\label{tm11}\label{tm5.13}
			Suppose that $f_2,f_1$ satisfy the conditions in Theorem~\ref{tm8} or Theorem~\ref{tm10}.  For each case, there exists a positive constant $c_{f_2,f_1}$ depending on $f_2,f_1$, and $C_2$ depending on $G$. Then the following results hold:
			\begin{enumerate}[(a)]
				\item Equation \eqref{eq:maineq} has a solution for $0\leqslant c \leqslant c_{f_2,f_1}$ in case $(a)$, and we conclude that 
				$$0<c_{f_2,f_1}\leqslant C_2\frac{\max_{x\in V} |f_2-\frac{1}{2}\epsilon f_{1,-}^2|}{\max_{x\in V} (f_2-\frac{1}{2}\epsilon f_{1,-}^2)}+\frac{1}{\epsilon},$$
				where we let $\epsilon=0, \frac{1}{\epsilon}=0$ for $f_1\geqslant0$, and for $\min_{x\in V}f_1<0$ define  
				$$\epsilon= \frac{\max_{x\in V}f_2}{\max_{x\in V}f_{1,-}^{2}}.
				$$
				
				\item Equation \eqref{eq:maineq} has a solution for $0\leqslant c\leqslant c_{f_2,f_1}$ in case $(c)$, and
				$$0<c_{f_2,f_1}\leqslant C_2\left(\ln \frac{\max_{x\in V}|f_1|^2+\max_{x\in V} f_{2,-}}{f_1(x_0)^2}+1\right),$$
				where $x_0$ is the vertex mentioned in case $(c)$.
				
				\item equation \eqref{eq:maineq} has a solution for $ -c_{f_2,f_1}\leqslant c\leqslant0$ in case $(d)$, and 
				$$0<c_{f_2,f_1}	\leqslant \max_{x\in V,f_2(x)\neq 0}\frac{f_1^2}{4f_{2,-}}.$$		  	
			\end{enumerate}
			\tmd
			\begin{proof}
				In this proof, we use a positive constant $C_2$ that varies but has a uniform upper bound which depends only on $G$.
				
				In case $(a)$, we use the methods in \cite{ge2017kazdan,sun2022brouwer} to prove the estimates of $c_{f_2,f_1}$. Suppose $u$ is a solution to \eqref{eq:maineq}. Since $\Delta e^u\geqslant e^u\Delta u$, we have
				$$\left(\Delta-2c\right) e^{-2u}\geqslant 2f_2+2f_1e^{-u}.$$
				If $\min_{x\in V}f_1<0$, we suppose $c$ sufficiently large, and define $c'=c-\epsilon>0$. Note that $2f_1e^{-u}\geqslant -\epsilon f_{1,-}^2-\frac{1}{\epsilon}e^{-2u}$, we obtain
				$$\left(\Delta-2c'\right) e^{-2u}\geqslant 2f_2-\epsilon f_{1,-}^2:=H.$$
				Let $\xi$ be the solution of 
				$$\left(\Delta-2c'\right)\xi=H.$$
				Then by the maximum principle, 
				$$\xi\geqslant e^{-2u}>0.$$
				Thus assuming $c'$ larger than the largest eigenvalue of $-\Delta$,  we have
				\begin{equation*}
					\begin{aligned}
						0>-2\xi c'=\left(1-\frac{\Delta}{2c'}\right)^{-1}H=H+(2c')^{-1}\Delta H+O\left(c'^{-2}\max_{x\in V}|\Delta H|\right).
					\end{aligned}
				\end{equation*}
				Considering on the vertex where $H$ attains its maximum, this yields that
				$$c'\lesssim \frac{\max_{x\in V}|\Delta H|}{\max_{x\in V} H}\lesssim \frac{\max_{x\in V}| H|}{\max_{x\in V} H}.$$
				If $f_1\geqslant0$, since 
				$$\left(\Delta-2c\right) e^{-2u}\geqslant 2f_2+2f_1e^{-u}\geqslant 2f_2,$$
				we obtain the similar results without $\epsilon$, and one can check the proof.
				
				In case $(c)$, we let
				$$c=c_1t,\ t=\frac{\max_{x\in V}f_{2,-}}{\max_{x\in V}f_1^2 }>0.$$
				On the vertex where $u$ attains its minimum, we get
				$$-\max_{x\in V} f_{2,-}e^{2\min_{x\in V}u}-\max_{x\in V}|f_1|e^{\min_{x\in V}u}+c\leqslant 0.$$
				This  implies that
				\begin{equation*}
					\begin{aligned}
						e^{\min_{x\in V}u}\geqslant \frac{\max_{x\in V}|f_1|}{\max_{x\in V} f_{2,-}}\cdot \frac{\sqrt{1+4c_1t^2}-1}{2}.
					\end{aligned}
				\end{equation*}
				Note that on the vertex $x_0$, we have
				$$f_1(x_0)e^u+c=-\Delta u(x_0)\leqslant C_2\left(u(x_0)-\min_{x\in V}u\right).$$ 
				Let $l(x)= C_2x-f_1(x_0)e^x$, and one easily sees that
				$$\sup_{x\in \mathbb{R}}l(x)=l\left(\ln \frac{C_2}{f_1(x_0)}\right)=C_2\left(\ln \frac{C_2}{f_1(x_0)}-1\right).$$
				Thus, 
				\begin{equation*}
					\begin{aligned}
						c\leqslant &C_2\left[\ln \frac{C_2}{f_1(x_0)}-1-\ln \left(\frac{\max_{x\in V}|f_1|}{\max_{x\in V} f_{2,-}}\cdot \frac{\sqrt{1+4c_1t^2}-1}{2}\right)\right]\\
						& =C_2 \ln \left(\frac{C_2}{ef_1(x_0)}\cdot\frac{\max_{x\in V}f_{2,-}}{\max_{x\in V}|f_1| }\cdot \frac{2}{\sqrt{1+4c_1t^2}-1}\right)\\
					\end{aligned}
				\end{equation*}
				If $c_1t^2\geqslant1$, we have $\sqrt{1+4c_1t^2}\geqslant 1+\sqrt{c_1}t=1+\sqrt{ct}$, and obtain
				\begin{equation}\label{estimate:c1}
					\begin{aligned}
						\frac{C_2}{2}\ln c+c\leqslant \frac{C_2}{2}\ln \left(\frac{4C_2^2}{e^2}\frac{\max_{x\in V} f_{2,-}}{f_1(x_0)^2}\right).
					\end{aligned}
				\end{equation}
				If $c_1t^2<1$, since $$\frac{2}{\sqrt{1+4c_1t^2}-1}=\frac{\sqrt{1+4c_1t^2}+1}{2c_1t^2}<\frac{2}{ct},$$
				we derive that
				\begin{equation}\label{estimate:c2}
					\begin{aligned}
						C_2\ln c+c\leqslant C_2\ln \left(\frac{2C_2}{e}\frac{\max_{x\in V} |f_1|}{f_1(x_0)}\right).
					\end{aligned}
				\end{equation}
				Note that $\frac{\max_{x\in V} |f_1|}{f_1(x_0)}\geqslant1$.  Suppose $c_{f_2,f_1}\geqslant1$, and one deduces the results by combining \eqref{estimate:c1} and \eqref{estimate:c2}.

				In case $(d)$, on any vertex where $f_2=0$, we have
				$$-\Delta u=f_1e^u+c\leqslant0.$$
				If we suppose
				$$c<-\max_{x\in V,f_2(x)\neq 0}\frac{f_1^2}{4f_{2,-}},$$
				on vertices where $f_2<0$, we have
				$$-\Delta u=f_2\left(e^u+\frac{f_1}{2f_2}\right)^2-\frac{f_1^2}{4f_2}+c<0,$$
				which contradicts with $\sum_{V}\Delta u=0$.
				
				Finally, letting $c\rightarrow c_{f_2,f_1}$ in cases $(a),(c)$ and $c\rightarrow-c_{f_2,f_1}$ in case $(d)$, by Theorem~\ref{tm1} we get the convergence of the solution sequence. Thus for $c= c_{f_2,f_1}$ in cases $(a),(c)$ and $c=-c_{f_2,f_1}$ in case $(d)$, there also exists at least one solution. For the case $c=0$, we consider $c_1>0$ in case $(a)$ or $(c)$, which corresponds to a solution $u_1$ as a supersolution. Then we repeat the method in Theorem~\ref{tm9} and know the existence of subsolutions. In case $(d)$, one easily sees that by choosing $c_1<0$ with corresponding solution $u_1$, $u_1$ serves as a subsolution, while a sufficiently large positive constant can serve as a supersolution. These yield the existence of at least one solution.
			\end{proof}
			
			\begin{remark}
				If $u$ is a solution of \eqref{eq:maineq}, we let 
				$$g_2=f_2e^{2t},\ g_1=f_1e^t.$$
				Then $u-t$ is a solution of the equation
				$$-\Delta v=g_2e^{2v}+g_1e^v+c.$$
				This explains why the term $\frac{f_2}{f_1^2}$ appears in Theorem~\ref{tm11}.
			\end{remark}
			
			Now we have proved the existence results in all cases. Recall that in Subsection~\ref{sub5.1}, we discuss the number of solutions when the Brouwer degree is equal to zero. Suppose that $f_2,f_1$ satisfy the conditions in Theorem~\ref{tm8} or Theorem~\ref{tm10}, and by Theorem~\ref{tm5.13} we know that the equation has at least one solution if $0\leqslant c\leqslant c_{f_2,f_1}$ in case $(a)$ or $(c)$, or $0\geqslant c\geqslant-c_{f_2,f_1}$ in case $(d)$, which is a local minimum of $J_{f_2,f_1,c}$. Although we are not yet sure whether $J_{f_2,f_1,c}$ is nondegenerate on this critical point, we can prove multiple solutions results for $|c|\neq c_{f_2,f_1}$ and $c\neq0$ by the method of \cite{li2024topological}.
			
			\tm\label{tm12}\label{tm5.15}
			Suppose that $f_2,f_1$ satisfy the conditions in Theorem~\ref{tm8} or Theorem~\ref{tm10} and $c_{f_2,f_1}$ is the constant defined in Theorem~\ref{tm5.13}. If $0<c<c_{f_2,f_1}$ in case $(a)$ or $(c)$, or $0>c>-c_{f_2,f_1}$ in case $(d)$, equation \eqref{eq:maineq} will have at least two solutions.
			\tmd
			
			\begin{proof}
				We prove this theorem by contradiction. By the discussion in Subsection~\ref{sub5.1}, we can assume that $u_0$ is the unique solution of \eqref{eq:maineq}. First, we claim that $u_0$ is an isolated local minimum. Let $c_1$ be a constant satisfying $c<c_1<c_{f_2,f_1}$ in cases $(a),(c)$ and $c>c_1>-c_{f_2,f_1}$ in case $(d)$. Thus there exists a solution $u_1$ of 
				$$-\Delta u=f_2e^{2u}+f_1e^u+c_1.$$
				Considering the minimizing problem and applying the method of sub- and supersolutions, we derive that $u_0$ is a local minimum of $J_{f_2,f_1,c}$ in $\{-M\leqslant u\leqslant u_1\}$ for cases $(a),(c)$, and in $\{u_1\leqslant u\leqslant M\}$ for case $(d)$, where $M$ is a sufficiently large constant. Note that for cases $(a),(c)$, if there exists $u_0(x')=u_1(x')$, we have
				$$-\Delta\left(u_0(x')-u_1(x')\right)=c-c_1<0.$$
				This implies that there exists some vertex $y'$ satisfying $u_0(y')-u_1(y')>0$, which leads to a contradiction. Thus, $u_0<u_1$ and there exists some $\epsilon_0>0$ such that 
				$$\inf_{u_0-\epsilon_0<u<u_0+\epsilon_0<u_1}J_{f_2,f_1,c}=J_{f_2,f_1,c}(u_0).$$   
				Since a local minimum of $J_{f_2,f_1,c}$ in $\{u_0-\epsilon_0<u<u_0+\epsilon_0<u_1\}$ must be a solution of \eqref{eq:maineq}, we obtain $u_0$ is an isolated local minimum.
				
				Let
				$$\widetilde{J}=\{J_{f_2,f_1,c}(u)\leqslant J_{f_2,f_1,c}(u_0)\}.$$
				According to \cite{chang1993infinite} (Chapter 1, Page 32), we define the $q-th$ critical group by
				$$C_q(J_{f_2,f_1,c},u_0)=H_q\left(\widetilde{J}\cap U, \left(\widetilde{J}\backslash \{u_0\}\right)\cap U, G\right),$$
				where $U$ is a neighborhood of $u_0$ and $H_{*}(X,Y;G)$ stands for the singular relative homology groups with coefficient group G of $J_{f_2,f_1,c}$ at $u_0$, say $\mathbb{Z}, \mathbb{R}$. Using the excision property of the singular homology theory in \cite{chang1993infinite}, the $q-th$ critical group is well-defined and doesn't depend on the choice of $U$. Choose 
				$$U=\{u: u_0-\epsilon_0<u<u_0+\epsilon_0\},$$
				and by the above discussion one easily sees that
				$$\widetilde{J}\cap U=\{u_0\},\ \left(\widetilde{J}\backslash \{u_0\}\right)\cap U=\varnothing.$$
				Thus from the definition of the singular relative homology groups, we obtain that
				$$C_q(J_{f_2,f_1,c},u_0)=H_q\left(\{u_0\},G\right)=\delta_{0}G.$$
				Note that if $J_{f_2,f_1,c}(u_i)\rightarrow c_0\in \mathbb{R}$ and $J_{f_2,f_1,c}'(u_i)\rightarrow 0$, we can suppose that for large $i$, 
				$$-\Delta u_i=f_2e^{2u_i}+f_1e^{u_i}+c+h_i,$$
				where $|h_i|<\frac{c}{2}$ and $h_i\rightarrow0$.
				In view of Theorem~\ref{tm1}, we deduce that $\{u_i\}$ is a uniformly bounded sequence, which leads to a convergent subsequence. These yield that $J_{f_2,f_1,c}$ satisfies the $(PS)$ condition. Applying Theorem 3.2 in \cite{chang1993infinite} (Chapter 2, Page 100), we finally get that
				$$d_{f_2,f_1,c}=\deg(DJ_{f_2,f_1,c}, B_R,0)=\sum_{q=0}^\infty (-1)^q\rank\ C_q(J_{f_2,f_1,c},u_0)=1,$$
				which contradicts with $d_{f_2,f_1,c}=0$.
			\end{proof}
			
			\section{General exponential nonlinearity}\label{se6}
			In the final section, we discuss the semilinear equations of general exponential nonlinearity. By using the same methods in Section~\ref{se3} and \ref{s4}, we can obtain a priori estimates and compute the corresponding topological degree. Consider equation \eqref{eq:maineq}
			\begin{equation*}
				\begin{aligned}
					-\Delta u=\sum_{i=1}^nf_ie^{iu}+c,
				\end{aligned}
			\end{equation*}
			where $f_n$ is not always equal to zero. Actually, following the proof in previous sections, one can check that all theorems in Section~\ref{se5} hold similarly for general exponential nonlinearity. In this section, we only list the results and provide a brief proof, and repeated proof steps will be omitted.
			
			\subsection{A Priori estimates}
			We consider a priori estimates of equation \eqref{eq:maineq}. Applying the methods in Theorem~\ref{tm1} and \ref{tm2}, one easily sees that the similar a priori estimates hold.
			
			\begin{proof}[Proof of Theorem~\ref{tm6.1}]
				The proof strategy is the same as Theorem~\ref{tm1} and Theorem~\ref{tm2}, and we only emphasize some important details. Assume that in each case, by passing to subsequences, there exist sequences $\{f_{i,m}\}$, $1\leqslant i\leqslant n$, $\{c_m\}$, satisfying
				$$\lim_{m\rightarrow+\infty}f_{i,m}=f_i,\ \lim_{m\rightarrow+\infty}c_m=c,$$
				where $\overline{f_{1,m}}$ have the same sign with respect to $m$, and $|c|\geqslant K^{-1}$ or $c=0$. The corresponding solution sequence $\{u_m\}$ is not uniformly bounded. By Theorem~\ref{tm:alternative}, we know that $u_m$ tends to $-\infty$ or $+\infty $ uniformly.

				Step1: Assume that $u_m\rightarrow-\infty$ uniformly. Then $|\Delta u_m|$ is uniformly bounded. If $c\neq0$, note that
				$$\lim_{m\rightarrow+\infty}\left(\sum_{i=1}^nf_{i,m}(x)e^{iu_m}+c_m\right)=c\neq0,$$
				which is a contradiction to $\sum_V\Delta u_m=0$. If $c=0$, since 
				$$\Delta e^{-u_m}\geqslant -e^{-u_m}\Delta u_m=\sum_{i=1}^{n-1}f_{i+1,m}e^{iu_m}+f_{1,m},$$
				we sum it over $V$ and let $m\rightarrow+\infty$, and will obtain a contradiction if $\overline{f_{1,m}}\geqslant K^{-1}$. For $\overline{f_{1,m}}\leqslant-K^{-1}$, we assume that for large $m$, 
				$$2\overline{f_{1}}\leqslant\overline{\sum_{i=1}^{n-1}f_{i+1,m}e^{iu_m}+f_{1,m}}\leqslant-\frac{1}{2}K^{-1}.$$  
				Using a priori estimates in \cite{sun2022brouwer} we know $\{u_m\}$ is uniformly bounded, which is a contradiction.
				
				Step2: Assume that $u_m\rightarrow+\infty$ uniformly. 
				
				For case $(i)$, we can derive a similar inequality to \eqref{ineq:main1}, and consider on the vertex $x_0$. Let $m\rightarrow+\infty$ and one easily sees a contradiction.
				
				For case $(ii)$, since the right hand side of \eqref{eq:maineq} is uniformly bounded above, we obtain $|\Delta u_m|$ is uniformly bounded. For each $m$, we choose a vertex $x_m$ where $D_{F_n}(x_m)\leqslant-K^{-1}$. Since 		     $$\lim_{m\rightarrow+\infty}\left(\sum_{i=1}^nf_{i,m}(x_m)e^{iu_m(x_m)}+c_m\right)=-\infty,$$
				we obtain a contradiction with the uniform bound of $|\Delta u_m|$.
			\end{proof}
			
			\subsection{Topological degree}
			Recall that we define
			$$H_{f_n,\dots,f_1,c}(u)=\Delta u+\sum_{i=1}^nf_ie^{iu}+c.$$
			For convenience, we write
			$H_{n,c}=H_{f_n,\dots,f_1,c}$ and let
			$$d_{n,c}=\lim_{R\rightarrow+\infty} \deg \left(H_{n,c},B_R,0\right).$$
			Consider the following cases.
			\begin{enumerate}[($a^*$)]
				\item There exists some $x_0\in V$ satisfying $D_{F_n}(x_0)>0$.
				\item $f_i\leqslant0$ for any integer $1\leqslant i\leqslant n$.             	
				\item $D_{F_n}(x)\leqslant0$. There exists some $f_i$ which is positive somewhere.
			\end{enumerate}
			For $0\leqslant k\leqslant n-2$, We denote by
			$$A_k=\{x\in V: f_n(x)=f_{n-1}(x)=\dots=f_{n-k}(x)=0\}.$$
			Without losing generality, we suppose that $m\equiv1$. Applying the methods in Section~\ref{s4}, we can obtain the following theorems.
			
			\tm\label{tm6.2}
			Suppose that $u$ is a solution to \eqref{eq:maineq} and $c\neq 0$. Then we have
			\begin{equation*}
				\begin{aligned}
					d_{n,c}=\left\{
					\begin{aligned}
						1\ ,\ \ \   &c>0\ \text{in}\ (b^*)\ \text{and}\ (c^*),\\
						-1,\ \ \ &c<0\ \text{in}\ (a^*),\\
						0,\ \ \ &\text{otherwise}.\\
					\end{aligned}
					\right.
				\end{aligned}
			\end{equation*}
			\tmd
			
			\begin{proof}
				The proof strategy is the same as Theorem~\ref{tm3}, and we only need to note that in case $(c^*)$, let
				$$f_{n,t}(x)=\left\{
				\begin{aligned}
					&(1-t)f_n(x)-1\ ,  &f_n(x)<0,\\
					&\ 0, &f_n(x)=0,\\
				\end{aligned}
				\right.$$
				and for $0\leqslant k\leqslant n-2$, 
				$$f_{n-k-1,t}(x)=\left\{
				\begin{aligned}
					(1-t)f_{n-k-1}(x)-1\ ,\ \ \   &x\in A_k \ \text{and}\ f_{n-k-1}(x)<0,\\
					(1-t)f_{n-k-1}(x)\ \ \ \ \ \  ,\ \ \ &\text{otherwise}.\\
				\end{aligned}
				\right.$$
				Then we return back to case $(b^*)$, and one can easily repeat the proof and compute the topological degree.
			\end{proof}
			
			\tm\label{tm6.3}
			Suppose that $u$ is a solution to \eqref{eq:maineq} and $\overline{f_1}\neq 0,\ c= 0$. Then we have
			\begin{equation*}
				\begin{aligned}
					d_{n,c}=\left\{
					\begin{aligned}
						1 &, \  \ \overline{f_1}>0 \ \text{in}\ (c^*),\\
						-1&, \ \  \overline{f_1}<0 \ \text{in}\ (a^*),\\
						0&, \ \ \text{otherwise}.
					\end{aligned}
					\right.
				\end{aligned}
			\end{equation*}
			\tmd
			
			\begin{proof}
				For each case, we do the same homotopy transformation on $f_i$ for $2\leqslant i\leqslant n$ as in Theorem~\ref{tm6.2}. Theorem~\ref{tm6.1} requires that we should keep $\overline{f_1}$ away from zero when doing homotopy transformations. Thus, in case $(a^*)$, we transform $f_1$ as case $(c)$ in Theorem~\ref{tm4} if $D_{F_n}(x_0)$ is exactly $f_1(x_0)$. Otherwise, we let 
				$f_{1,t}=
				(1-t)f_{1}(x)-1$
				when $\overline{f_1}<0$, and $f_{1,t}=
				(1-t)f_{1}(x)+1$
				when $\overline{f_1}>0$.             
				In case $(b^*)$, equation \eqref{eq:maineq} must have no solutions since $-\Delta u\leqslant0$. In case $(c^*)$, we note that after homotopy transformations, for each vertex $x\in V$, there exists at most one function $f_i$ with a non-zero value for $2\leqslant i\leqslant n$, such that $f_i(x)=-1$. If this $f_i$ exists, we label this index $i$ as $i(x)$. Then when $\overline{f_1}<0$ we let
				$$f_{1,t}(x)=\left\{
				\begin{aligned}
					&(1-t)f_{1}(x)\ ,\ \ \   &x\in A_{n-2}, \\
					&(1-t)f_{1}(x)-1,\ \ \ &x\notin A_{n-2},\\
				\end{aligned}
				\right.$$
				and one easily sees that $f_{i,1}\leqslant0$, which implies that $d_{n,c}=0$. When $\overline{f_1}>0$ we let
				$$f_{1,t}(x)=\left\{
				\begin{aligned}
					&(1-t)f_{1}(x)\ ,\ \ \   &x\in A_{n-2}, \\
					&(1-t)f_{1}(x)+1,\ \ \ &x\notin A_{n-2}.\\
				\end{aligned}
				\right.$$
				By applying Lemma~\ref{lm3}, we remove the vertices where $f_{n,1}=f_{n-1,1}=\dots=f_{1,1}=0$, and obtain a new equation:
				$$-\Delta u(x)=-e^{i(x)u}+e^u.$$
				This equation only has one solution $u=0$. Thus we obtain the topological degree by 
				$$d_{n,c}=\sgn \det \left(-\Delta+I^*\right)=1,$$
				where $I^*$ is a diagonal matrix with diagonal elements $i(x)-1\geqslant1$.
			\end{proof}
			
			\subsection{Existence results}
			The Kronecker existence theorem implies the existence of solutions of \eqref{eq:maineq} if $d_{n,c}\neq0$. In addition, we easily see that the equation has no solutions when $c\leqslant0$ in case $(b^*)$. Thus we only discuss the cases:
			\begin{enumerate}[($A^*$)]
				\item $c\geqslant 0$ in cases $(a^*)$.
				\item $c\leqslant 0$ in cases $(c^*)$.
			\end{enumerate}
			By the same process in Subsection~\ref{5.2} and the method of sub- and supersolutions, we know that $u=-M\ (\text{resp.}\ M)$ is a subsolution $(\text{resp. supersolution})$ of \eqref{eq:maineq} in case $(A^*)\ (\text{resp. }\ (B^*))$, where $M$ is a sufficient large positive constant. There also exists a $c_{f_1,\dots,f_n}$, which may be zero, a positive constant or infty, such that the equation has a solution when $c\in (0,c_{f_1,\dots,f_n}]$ in case $(A^*)$, $c\in [-c_{f_1,\dots,f_n},0)$ in case $(B^*)$. For convenience, we write $c_n=c_{f_1,\dots,f_n}$.
			
			First we have the following theorem.
			
			\tm\label{tm6.4}
			Suppose that $\overline{f_1}<0$ in case $(A^*)$, and $\overline{f_1}>0$ in case $(B^*)$. Then equation \eqref{eq:maineq} has at least one solution when $c$ is sufficiently close to zero. 
			\tmd
			
			\begin{proof}
				We choose a function $v:V\rightarrow \mathbb{R}$ which satisfies
				$$\Delta v=\sum_{i=1}^na^{i-1}\left(f_i-\overline{f_i}\right),$$
				where $a>0$ sufficiently small. The remaining proof steps are the same as Theorem~\ref{tm8}, and one can check the details.
			\end{proof}
			
			For the rest situations, we need to  add some restrictions to $f_i$, $1\leqslant i\leqslant n$, and we can derive the existence of solutions.
			
			\tm\label{tm6.5}
			We suppose that  $\overline{f_1}\geqslant0$ in case $(A^*)$, and   $\overline{f_1}\leqslant0$ in case $(B^*)$. For each case, there exists $f_i^*$, $2\leqslant i\leqslant n$, which depends on $f_1,f_2,\dots, f_{i-1},f_{i+1},\dots,f_n$. Then if $c$ is sufficiently close to zero, we derive that
			\begin{enumerate}[(a)]
				\item  Equation \eqref{eq:maineq} has a solution when $f_j\leqslant f_j^*$ in cases $(a^*)$, where $j$ is a integer satisfying $2\leqslant j\leqslant n$.
				\item Equation \eqref{eq:maineq} has a solution when $f_j\geqslant f_j^*$ in case $(c^*)$, where $j$ is a integer satisfying $2\leqslant j\leqslant n-1$.
			\end{enumerate}
			\tmd
			
			\begin{proof}
				Consider the following equation:
				$$-\Delta u=He^u,$$
				where $H$ changes sign and $\overline{H}<0$. This equation has a solution $u^*$. We define
				$$\epsilon_0=\frac{1}{2}\min_{H(x)\neq0 \atop x\in V }|H(x)|,\ \epsilon=\epsilon_0\min_{x\in V}e^{u^*(x)}.$$
				Note that replace $u$ with $u-k$, where $k$ is a constant, and equation \eqref{eq:maineq} is equivalent to 
				$$-\Delta u=\sum_{i=1}^nf_ie^{-ik}e^{iu}+c.$$
				In cases $(A^*)$, we only need to construct a supersolution, and the subsolution is constructed by using the methods in the proof of Theorem~\ref{tm9}. In case $(B^*)$, we need to construct a  subsolution. 
				
				In case $(A^*)$, let
				$$ f_j^*=e^{jk}e^{-(j-1)u^*}\left(H-\epsilon_0-\sum_{\substack{i=1\\i\neq j}}^{n}f_ie^{-ik}e^{(i-1)u^*}\right),$$
				and in case $(B^*)$, we let
				$$ f_j^*=e^{jk}e^{-(j-1)u^*}\left(H+\epsilon_0-\sum_{\substack{i=1\\i\neq j}}^{n}f_ie^{-ik}e^{(i-1)u^*}\right).$$
				In case $(A^*)$, we consider $f_j\leqslant f_j^*$ and for $c\leqslant \epsilon$, we obtain
				$$-\Delta u^*=He^{u^*}\geqslant(H-\epsilon_0)e^{u^*}+c\geqslant \sum_{i=1}^nf_ie^{-ik}e^{iu}+c,$$ 
				which implies that $u^*$ is a supersolution. In fact, by letting $k$ sufficiently large and choosing different bounded $H$, we can claim the existence of $f_j^*$, which can change sign and its symbols can be freely selected. By the same method, we can construct a supersolution for $f_j\geqslant f_j^*$ in case $(c^*)$.
				
			\end{proof}
			
			\begin{remark}
				We need to point out that in Theorem~\ref{tm6.5}, for case $(c^*)$, the reason we do not take $j=n$ is that we can consider the following example.
				
				Choose $f_i$ $(1\leqslant i\leqslant n-1)$ such that 
				$$\sum_{i=1}^{n-1}f_ie^{it}<0$$
				holds for any $t$. Thus, 
				$$-\Delta u=\sum_{i=1}^{n}f_ie^{iu}+c< 0,$$
				which implies the nonexistence of solutions, and it is independent of the choice of $f_n$ in case $(c^*)$.
			\end{remark}
			
			Combining Theorem~\ref{tm6.4} and ~\ref{tm6.5}, we derive that $c_n>0$ if $f_i$ satisfies the conditions in one of these two theorems. Under these assumptions, in the following we prove that $c_n<+\infty$, and if $c\in (0,c_n)$ in case $(A^*)$ or $c\in (-c_n,0)$ in case $(B^*)$, equation \eqref{eq:maineq} must have at least two solutions.
			
			\tm\label{tm6.6}
			Suppose that for $1\leqslant i\leqslant n$, $f_i$ satisfies the conditions in Theorem~\ref{tm6.4} or Theorem~\ref{tm6.5}. Then there exists a positive constant $c_n$ depending on $f_i$, $G$, and we have the following results:
			\begin{enumerate}[(a)]
				\item In case $(A^*)$, equation \eqref{eq:maineq} has at least two  solutions if $0< c< c_n$, has at least one solution if $c=0$ or $c=c_n$, and has no solutions if $c>c_n$.
				\item In case $(B^*)$, equation \eqref{eq:maineq} has at least two  solutions if $-c_n< c< 0$, has at least one solution if $c=0$ or $c=-c_n$, and has no solutions if $c<-c_n$. 			  	
			\end{enumerate}
			
			\tmd
			
			\begin{proof}
				Based on the discussion at the end of the proof of Theorem~\ref{tm11}, we only need to prove that $c_n<+\infty$ and the multiple solutions results. First we suppose that $c_n=+\infty$, which implies that for any $c>0$ in case $(a^*)$, any $c<0$ in case $(c^*)$, equation \eqref{eq:maineq} is solvable. 
				
				In case $(a^*)$, consider sufficiently large $c$, and we will derive a solution $u$. On the vertex $x_0$ where $u$ attains its minimum, we note that
				$$-K\sum_{i=1}^ne^{iu}+c\leqslant \sum_{i=1}^nf_ie^{iu}+c=-\Delta u(x_0)\leqslant0,$$
				where 
				$$K=\max_{x\in V \atop 1\leqslant i\leqslant n}|f_i(x)|.$$
				Let $c>nK$, we obtain that
				$$\min_{x\in V}u(x)\geqslant 0.$$
				By the assumptions in case $(a^*)$, we know there exists some $1\leqslant i\leqslant n$ and $x_1\in V$ such that
				$$-\Delta u(x_1)=f_ie^{iu}+\dots+f_1e^u+c,$$
				where $f_i(x_1)>0$. Since
				$$-\Delta u(x_1)=\frac{1}{m(x_1)}\sum_{y\in V}\omega_{x_1y}(u(x_1)-u(y))\leqslant \frac{1}{m(x_1)}\sum_{y\in V}\omega_{x_1y}u(x_1).$$
				we let $c\rightarrow+\infty$, and one easily sees the contradiction.
				
				In case $(c^*)$, the assumptions in this case imply that
				$$ \sum_{i=1}^nf_ie^{iu}\leqslant K_n,$$
				where $K_n$ is a constant independent of $u$. Let $c<-K_n$, and we deduce that
				$$-\Delta u\leqslant K_n+c<0,$$
				which derives that equation \eqref{eq:maineq} has no solutions.
				
				Applying the methods in Theorem~\ref{tm5.15} and repeating the similar process, we can prove the  multiple solutions results for $0< c< c_n$ in case $(A^*)$ or $-c_n< c< 0$ in case $(B^*)$.
			\end{proof}

			\bigskip
			\bigskip
			\textbf{Conflicts of Interests.} The authors declared no potential conflicts of interests with respect to this article.
			\bigskip

			\bibliographystyle{alpha}
			\bibliography{ckwx}

			\end{document}